\documentclass[12pt]{article}
\usepackage{amsthm,amsfonts,amssymb,amsmath}
\usepackage{enumerate}
\usepackage[colorlinks=true]{hyperref}

\newcommand{\BA}{{\mathbb {A}}}

\newcommand{\BC}{{\mathbb {C}}}

\newcommand{\BG}{{\mathbb {G}}}

\newcommand{\BN}{{\mathbb {N}}}

\newcommand{\BP}{{\mathbb {P}}}
\newcommand{\BQ}{{\mathbb {Q}}}
\newcommand{\BR}{{\mathbb {R}}}

\newcommand{\BZ}{{\mathbb {Z}}}

\newcommand{\CL}{{\mathcal {L}}}
\newcommand{\CM}{{\mathcal {M}}}

\newcommand{\CO}{{\mathcal {O}}}

\newcommand{\CX}{{\mathcal {X}}}
\newcommand{\CY}{{\mathcal {Y}}}

\newcommand{\RN}{{\mathrm {N}}}

\newcommand{\an}{{\mathrm{an}}}

\newcommand{\intb}{{\mathrm{int}}}

\newcommand{\Div}{{\mathrm{Div}}}
\renewcommand{\div}{{\mathrm{div}}}

\newcommand{\Gal}{{\mathrm{Gal}}}

\newcommand{\NS}{{\mathrm{NS}}}

\newcommand{\Pic}{\mathrm{Pic}}
\newcommand{\Prep}{\mathrm{Prep}}

\DeclareMathOperator{\Spec}{Spec}

\newcommand{\tor}{{\mathrm{tor}}}

\newcommand{\wh}{\widehat}
\newcommand{\pp}{\frac{\partial\bar\partial}{\pi i}}
\newcommand{\pair}[1]{\langle {#1} \rangle}

\newcommand{\ds}{\displaystyle}

\newcommand{\lra}{\longrightarrow}

\newcommand{\amp}{\mathrm{amp}}
\newcommand{\nef}{\mathrm{nef}}
\renewcommand{\vert}{\mathrm{vert}}

\newcommand{\CLL}{\overline{\mathcal L}}
\newcommand{\CMM}{\overline{\mathcal M}}

\newcommand{\CNN}{\overline{\mathcal N}}

\newcommand{\CC}{\mathbb{C}}
\newcommand{\RR}{\mathbb{R}}
\newcommand{\ZZ}{\mathbb{Z}}
\newcommand{\QQ}{\mathbb{Q}}

\newcommand{\BFF}{\mathbf{F}}

\newcommand{\DS}{\mathcal{DS}}

\newtheorem{thm}{Theorem}[section]
\newtheorem{cor}[thm]{Corollary}
\newtheorem{lem}[thm]{Lemma}
\newtheorem{prop}[thm]{Proposition}

\newtheorem{defn}[thm]{Definition}

\theoremstyle{definition}
\newtheorem{definition}[thm]{Definition}
\newtheorem{example}[thm]{Example}

\newtheorem{question}[thm]{Question}

\theoremstyle{remark}
\newtheorem{remark}[thm]{Remark}

\begin{document}

\title{The arithmetic Hodge index theorem for adelic line bundles I: number fields}
\author{Xinyi Yuan and Shou-Wu Zhang}
\maketitle

\tableofcontents

\section{Introduction}

The  Hodge index theorem for  divisors on arithmetic surfaces  proved by  Faltings  \cite{Fal}  in 1984 and Hirijac \cite{Hr} in 1985 is 
one of the  fundamental results 
in Arakelov theory. For example, it was used to prove the the first case of Bogomolov's conjecture for curves in tori by Zhang 
\cite{Zh1}. 
In 1996, Moriwaki  \cite{Mo1} extended  the Hodge index theorem for codimension one cycles on high-dimensional arithmetic 
varieties,  and then confirmed  the codimension one case of the  arithmetic  standard conjecture proposed by
 Gillet and Soul\'e  in \cite{GS3}. Despite its fundamental importance in number theory and  arithmetic geometry, e.g, in 
 Gross--Zagier type formula for cycles on Shimura varieties,  the 
 high-codimensional case of the Gillet--Soul\'e conjecture is still wide open. 
 
The  aim of this series of two papers is to prove an {\em adelic version} of the Hodge index theorem for (still) codimension one cycles on 
 varieties over a finitely generated field $K$. Namely, these line bundles are equipped with metrics as limits of integral models of the structure morphisms. These matrices  naturally appear  in  algebraic dynamical systems and moduli
 spaces  
 of varieties.  Here we will give  two applications:

\begin{enumerate}[(1)]
\item the uniqueness part of the
Calabi--Yau theorem for metrized line bundles over non-archimedean analytic spaces,
\item a rigidity result of the sets of preperiodic points of polarizable endomorphisms of a projective variety over any field $K$. 
\end{enumerate}

The proof of our results uses Arakelov theory (cf. \cite{Ar, GS1}) and Berkovich analytic spaces (cf. \cite{Be}). In comparison with Moriwaki's proof, one essential difficulty in adelic case is the lack of relative ampleness 
in projective systems of  integral models. 
 Our new ideas for this are  a new notion of $\bar L$-boundedness 
 and a variation method (inspired by Blocki's work \cite{Bl} in complex geometry).

In this paper, the first one of the series, we prove our Hodge index theorem and 
the application in (2) assuming $K$ is a number field, and prove the application in (1) in the full generality. 
In \cite{YZ}, the second one of the series, we will treat our Hodge index theorem and the application in (2) in the full generality after introducing a theory of adelic line bundles on varieties over finitely generated fields. We separate the exposition into two papers due to the technicality of our theory over finitely generated fields. 
In the following, we state the main results of this paper.

\subsection{Arithmetic Hodge index theorem}

Let us first recall the arithmetic Hodge index theorem for Hermitian line bundles on arithmetic varieties. 
\begin{thm}[\cite{Fal, Hr, Mo1}] \label{hodge0}
Let $K$ be a number field, and $\pi:\CX\to \Spec O_K$ be a regular arithmetic variety, geometrically connected of relative dimension $n\geq 1$. 
Let $\overline \CM$ be a Hermitian line bundle on $\CX$, and $\overline \CL$ be an ample Hermitian line bundle on $\CX$. Assume that $\CM_K\cdot \CL_{K}^{n-1}=0$ on the generic fiber $\CX_K$. Then the arithmetic intersection number
$$
\overline \CM^2\cdot \overline \CL^{n-1}\leq 0.
$$ 

Moreover, if $\CL$ is ample on $\CX$ and the metric of $\CLL$ is strictly positive, then the equality holds if and only if $\overline \CM=\pi^*\overline \CM_0$ for some Hermitian line bundle $\overline \CM_0$ on $\Spec O_K$. 
\end{thm}
The result was due to Faltings \cite{Fal} and Hriljac \cite{Hr} for $n=1$, and due to Moriwaki \cite{Mo1} for general $n$. 

The main result of this paper is a version of the above result for adelic line bundles. The importance of the Hodge index theorem in the adelic setting will be justified by our applications to the Calabi--Yau theorem and to algebraic dynamics. 

To state our theorem, we start with some positivity notions.
We refer to Zhang \cite{Zh3} for basic definitions of adelic line bundles, and to Definition \ref{def-positivity2} for positivity of Hermitian line bundles on arithmetic varieties. 

\begin{defn} \label{def-positivity} 
Let $K$ be a number field. Let $X$ be a projective variety over $K$, and
$\overline L, \overline M$ be adelic line bundles on $X$.
We make the following definitions. 
\begin{enumerate} [(1)]
\item We say that $\overline L$ is \emph{nef} if the adelic metric of $L$ is a uniform limit of metrics induced by nef Hermitian line bundles on integral models of $X$. 
\item We say that $\overline L$ is \emph{integrable} if it is the difference of two nef adelic line bundles on $X$.  
\item We say that $\overline L$ is \emph{ample} if
$L$ is ample, $\overline L$ is nef, and $(\overline L|_Y)^{\dim Y+1}> 0$ for any closed subvariety $Y$ of $X$.
\item We say that $\overline M$ is \emph{$\overline L$-bounded} if there is an integer $m >0$ such that both $m\overline L+ \overline M$ and $m\overline L- \overline M$ are nef. 
\end{enumerate}
\end{defn}

 \medskip
 
The following is the main theorem of this paper.

\begin{thm}\label{hodge} 
Let $K$ be a number field, and $\pi:X\to \Spec K$ be a normal and geometrically connected projective variety of dimension $n\geq 1$. 
Let $\overline M$ be an integrable adelic line bundle on $X$, and $\overline L_1, \cdots, \overline L_{n-1}$ be $n-1$ nef line bundles on $X$ where each $L_i$ is big on $X$. 
Assume $M\cdot L_1\cdots L_{n-1}=0$ on $X$. 
Then 
$$\overline M^2\cdot \overline L_1\cdots \overline L_{n-1} \le 0.$$

Moreover, $\overline L_i$ is ample and $\overline M$ is $\overline L_i$-bounded for each $i$, then the equality holds if and only if  $r\overline M=\pi^*\overline M_0$ for some adelic line bundle $\overline M_0$ on $\Spec K$ and some integer $r>0$.
\end{thm}

If $M$ is numerically trivial, we can strengthen the theorem as follows. 
First, we can remove the condition ``$L_i$ is big'' in the inequality part of the theorem, by viewing $\overline L_i$ as the limit of $\overline L_i+\epsilon \overline A$ as $\epsilon\to 0$ for some ample $\overline A$.
Second, we can replace the condition ``$\overline L_i$ is ample'' by ``$L_i$ is ample'' in the equality part of the theorem, by changing the metric of $\overline L_i$ by constant multiples. 

As in the classical case, the theorem explains the signature of the intersection pairing on certain space of adelic line bundles. 
Let $W$ denote the subspace of $\wh \Pic(X)_{\intb}\otimes_\ZZ\QQ$ consisting of elements which are represented by $\QQ$-linear combinations of integrable adelic line bundles on $X$ which are $\overline L_i$-bounded for all $i$. 
Define a pairing on $W$ by
 $$\pair{\overline M_1, \overline M_2}=\overline M_1\cdot \overline M_2\cdot \overline L_1\cdots \overline L_{n-1}.$$
Denote $V=\pi^*\wh \Pic (K)\otimes_{\ZZ}\QQ$, viewed as a subspace of $W$. 
Then the theorem implies that the pairing on $V^\perp$ is negative semi-definite, that $V$ is a maximal isotropic subspace of $V^\perp$, and that  $V^\perp/V$ is negative definite.

The inequality part of Theorem \ref{hodge} can be eventually reduced to Theorem \ref{hodge0} by taking limits (in the case all $\overline L_i$ are equal). However, the equality part of the theorem is more profound and significantly more difficult, since it is not a simple limit of its counterpart for Hermitian line bundles on integral models. 
The following are some new ingredients of our treatment:
\begin{enumerate}[(1)]
\item We introduce the notion ``$\overline L$-bounded'' to overcome the difficulty that arithmetic ampleness of Hermitian line bundles is not preserved by pull-back by morphisms between integral models, and the difficulty that
arithmetic ampleness of adelic line bundles is not preserved by perturbations.
\item We use a variational method to show that, in the case that $M$ is numerically trivial, if $\overline M^2\cdot \overline L_1\cdots \overline L_{n-1} = 0$ holds, then it holds after replacing $\overline L_i$ by 
any nef adelic line bundle $\overline L_i^0$ with the same underlying line bundle $L_i$. It gives much flexibility. The variational method is inspired by the idea of Blocki \cite{Bl} on the complex Calabi--Yau theorem. 
\item We only assume that $X$ is normal and projective in order to cover all polarizable algebraic dynamical systems. To overcome the difficulty, we extend the classical Lefschetz theorems for line bundles to the normal case. 
\end{enumerate}

\begin{example}
We give some examples to show that the conditions of the theorem are necessary. 
\begin{enumerate}[(1)]
\item
The assumption ``$L_i$ is big'' is necessary for the inequality if $\overline M$ is not vertical. In fact, take $n=2$. Take $\overline L_1=\pi^* \overline N$ for any adelic line bundle 
$\overline N\in \wh\Pic(K)$
with $\wh\deg(\overline N)=1$. Note that $L_1=\CO_X$ is trivial and the constraint $M\cdot L_1=0$ is automatic. 
Then the inequality $\overline M^2\cdot\overline L_1\leq 0$ is just 
$M^2\leq 0$ for any line bundle $M$ on $X$. 
It is not true. 
\item 
The assumption ``$\overline M$ is $\overline L_i$-bounded for each $i$'' is necessary for the condition of the equality. 
In fact, take $n=2$. Take an integral model $\CX$ of $X$. 
Let $\CLL$ be an ample Hermitian line bundle on $\CX$. 
Let $\alpha:\CX'\to \CX$ be the blowing-up of a closed point on the $\CX$. 
Let $\CMM$ be the (vertical) line bundle on $\CX'$ associated to the exceptional divisor endowed with the trivial hermitian metric $\|1\|=1$. Then we  $\CM_K\cdot\CL_K=0$ and $\CMM^2\cdot \alpha^*\CLL=0$ by the projection formula. 
But $\CMM$ is not coming from any line bundle on the base $\Spec O_K$. 
The problem is that $\CMM$ is not $\CLL$-bounded if we convert the objects to the adelic setting. 
\item  
The assumption ``$\overline L_i$ is ample'' is necessary for the condition of the equality if $M$ is not numerically trivial. 
In fact, take $X$ to be an abelian variety of dimension $n$. 
The multiplication $f=[2]$ defines a polarizable algebraic dynamical system on $X$. Let $N_1,N_2$ be two symmetric and ample line bundles on $X$, which are not proportional.
They polarizes $f$ in the sense that $f^*N_i=4N_i$ for $i=1,2$.  
Let $\overline N_i$ be the $f$-invariant adelic line bundles extending $N_i$. 
Then we have $\overline N_1^j\cdot\overline N_2^{n+1-j}=0$ for any $j$ by $f^*\overline N_i=4\overline N_i$. 
It follows that 
$$
(\overline N_1-\overline N_2)^{2}\cdot (\overline N_1+\overline N_2)^{n-1}=0.
$$
We are close to the setting of the theorem with $\overline M=\overline N_1-\overline N_2$
and $\overline L_i=\overline N_1+\overline N_2$. 
We can even adjust $N_1$ by a multiple to make 
$$
(N_1- N_2)\cdot ( N_1+ N_2)^{n-1}=0.
$$
Note that $N_1,N_2$ are not equal, we cannot apply the equality part of the theorem. The problem is that $\overline N_1+\overline N_2$ is not ample. 
\end{enumerate}

\end{example}

\subsection{Calabi--Yau theorem}
When the line bundle $M$ is trivial, Theorem \ref{hodge} is essentially a result on projective varieties over local fields. 
It gives some cases of the non-archimedean Calabi--Yau Theorem. 
In fact, the proof generalizes to the following full generality.  

Let $K$ be either $\BC$ or a field with a non-trivial complete non-archimedean
 absolute value $|\cdot|$. 
Let $X$ be a geometrically connected projective variety over $K$, and $L$ be an ample line bundle on $X$ endowed with a continuous semipositive $K$-metric $\|\cdot\|$. Then $(L, \|\cdot\|)$ induces
a canonical semipositive measure $c_1(L, \|\cdot\|)^{\dim X}$ on the
analytic space $X^{\rm an}$. We explain it as follows.

If $K=\BC$, then $X^{\rm an}$ is just the complex analytic
space $X(\BC)$. The measure $c_1(L, \|\cdot\|)^{\dim X}$ is
just the determinant of the Chern form $c_1(L, \|\cdot\|)$ which is locally
defined by
$$\displaystyle c_1(L, \|\cdot\|)=\pp \log \|\cdot\|$$
in complex analysis.

If $K$ is non-archimedean, $X^{\rm an}$ is the Berkovich space associated to the
variety $X$ over $K$. It is a
Hausdorff, compact and path-connected topological space.
Furthermore, it naturally includes the set of closed points of $X$.
We refer to \cite{Be} for more details on the space.
The metric $\|\cdot\|$ being semipositive, in the sense of \cite{Zh3},
means that it is a uniform limit of metrics induced by ample
integral models of $L$. The canonical measure $c_1(L,
\|\cdot\|)^{\dim X}$ is constructed by Chambert-Loir \cite{Ch}
in the case that $K$ contains a countable and dense subfield, 
and extended to the general case by Gubler \cite{Gu2}.

\begin{thm} \label{calabiyau}
Let $L$ be an ample line bundle over $X$, and $\|\cdot\|_1$ and
$\|\cdot\|_2$ be two semipositive metrics on $L$. Then
$$
c_1(L, \|\cdot\|_1)^{\dim X}=c_1(L, \|\cdot\|_2)^{\dim X}
$$
if and only if $\ds\frac{\|\cdot\|_1}{\|\cdot\|_2}$ is a constant.
\end{thm}

The history of the theorem in the complex case is as follows.
In the 1950s, Calabi \cite{Ca1, Ca2} made the following famous conjecture:
\emph{Let $X$ be a compact complex manifold endowed with a Kahler form
$\omega$, 
and let $\Omega$ be a positive smooth volume form on $X$ such that 
$\int_X \Omega =\int_X \omega^{\dim X}$.
Then there exists a smooth real-valued function $\phi$ on $X$ such that
$(\omega+i\partial \overline\partial \phi)^{\dim X}=\Omega.$}
Calabi proved that the function $\phi$ is unique up to scalars (if it
exists).
The existence of the function is much deeper, and was finally solved by S. T.
Yau
in the seminal paper \cite{Ya} in 1977.
Now the whole results are called the Calabi--Yau theorem.

If $\omega$ is cohomologically equivalent to a line bundle $L$ on $X$, 
the results can be stated in terms of existence and uniqueness of metrics
$\|\cdot\|$ on $L$
with $c_1(L, \|\cdot\|)^{\dim X}=\Omega$. 

Theorem \ref{calabiyau} includes the non-archimedean analogue of the uniqueness
part of the Calabi--Yau theorem.
The ``if'' part of the theorem is trivial by definition. For archimedean $K$,
the positive smooth
case is due to Calabi as we mentioned above, and the continuous semipositive
case is due to Kolodziej \cite{Ko}.
Afterwards Blocki \cite{Bl} provided a very simple proof of Kolodziej's result.

The theorem will be proved analogously to the case of trivial $M$ of Theorem \ref{hodge}. It can be viewed as a local arithmetic Hodge index theorem. 
Both theorems are proved utilizing Blocki's idea.

\subsection{Algebraic dynamics}

Let $X$ be a projective variety over a field $K$. 
A \emph{polarizable algebraic dynamical system on $X$}
is a morphism $f:X\to X$ such that there is an ample $\QQ$-line bundle  
$L\in \Pic(X)\otimes_{\BZ} \BQ$ satisfying $f^*L =qL$ from some rational number $q>1$. We call $L$ a \emph{polarization of $f$}, and call the triple $(X,f,L)$ a \emph{polarized algebraic dynamical system}.
Let $\Prep (f)$ denote the set of \emph{preperiodic
points}, i.e.,
$$
\Prep(f):=\{x\in X(\overline K)\ |\ f^m(x)=f^n(x) {\rm\ for\ some\ } m,n\in \BN,
\ m\neq n \}.
$$
A well-known result of Fakhruddin \cite{Fak} asserts that $\Prep (f)$ is
Zariski dense in $X$.

Denote by $\DS(X)$ the set of all polarizable algebraic dynamical systems 
$f$ on $X$. Note that we do \emph{not} require elements of $\DS(X)$ to be polarizable by the same ample line bundle or have the same dynamical degree $q$.

\begin{thm}\label{dynamicsmain}
Let $X$ be a projective variety over a number field $K$.
For any $f,g\in \DS(X)$, the following are equivalent:
\begin{itemize}
\item[(1)] $\Prep(f)= \Prep(g)$;
\item[(2)] $g\Prep(f)\subset\Prep(f)$;
\item[(3)] $\Prep(f)\cap \Prep(g)$ is Zariski dense in $X$.
\end{itemize}
\end{thm}

The theorem will be proved over any field $K$ in our second paper \cite{YZ}.  The proof of the number field case in this paper gives a rough idea of our approach for the general case.  

In an early version of the series, we require $f$ and $g$ to be polarizable by the same ample line bundle. The proof combines the equidistribution theorem of Yuan \cite{Yu} generalizing that of Szpiro--Ullmo--Zhang \cite{SUZ} et al, and the Calabi--Yau theorem in Theorem \ref{calabiyau}.
But we have removed this restriction in the current version by introducing the arithmetic Hodge index theorem, which is more powerful than the Calabi--Yau theorem.

\begin{remark}
When $X=\BP^1$, the theorem is independently proved by M. Baker and L. DeMarco \cite{BD} over any field $K$ of characteristic zero during the preparation of this paper. Their proof also applies to positive characteristics. 
Their treatment for the number field case is the same as our treatment in the earlier version, while the method for the general case is quite
different.
\end{remark}

\begin{remark}
The following are some works related to the above theorem (when $K$ is a number field):
\begin{itemize}
\item[(1)] A. Mimar \cite{Mi} proved the theorem in the case that $X=\BP^1$.
\item[(2)] If $f$ is a Latt\`es map on $\BP^1$ or a
power map on $\BP^d$ induced by $(\BG_m)^d$, the theorem is implied by the
explicit description of
$g$ by S. Kawaguchi and J. H. Silverman \cite{KS}.
\item[(3)] In the case that $X=\BP^1$, C. Petsche, L.
Szpiro, and T. Tucker \cite{PST} found a further equivalent
statement in terms of heights and intersections.
\end{itemize}
\end{remark}

Now we explain some ingredients in our proof of Theorem \ref{dynamicsmain}. The hard part is to show that (3) implies (1). 

Let $f,g\in\DS(X)$. 
We first consider the case that $f$ and $g$ are polarizable by the same ample line bundle $L$. 
Then we have an $f$-invariant adelic line bundle $\overline L_f$ and a
$g$-invariant adelic line bundle $\overline L_g$. 
Apply the successive minima to the sum $\overline L_f+\overline L_g$. 
By the assumption that $\Prep(f)\cap \Prep(g)$ is Zariski dense, 
the essential minimum of $\overline L_f+\overline L_g$ is 0. It follows that $(\overline L_f+\overline L_g)^{n+1}=0$. 
By the nefness, it gives $\overline L_f^i\cdot \overline L_g^{n+1-i}=0$ for any $i$. 
Hence, we have 
$$(\overline L_f-\overline L_g)^{2}\cdot (\overline L_f+\overline L_g)^{n-1}=0.$$
Then we apply our arithmetic Hodge index theorem to conclude that 
$\overline L_f=\overline L_g$. 
It follows that $f$ and $g$ have the same canonical height. 
Then they have the same set of preperiodic points.

The general case that $f$ and $g$ are not polarizable by the same ample line bundle is much more difficult. We develop a theory of admissible adelic line bundles to overcome the difficulty. The idea is to write every line bundle as a sum of eigenvectors of $f^*$. For any ample class $\xi\in \NS(X)_\QQ$, 
we first prove that there is a unique ``$f$-admissible lifting''
$L\in \Pic(X)_\QQ$ of $\xi$.
Note that $L$ does not necessarily polarize $f$.  
Then we construct an $f$-admissible adelic line bundle $\overline L_f$ extending 
$L$. We prove that $\overline L_f$ is nef. 
Similarly, from $\xi$, we have a unique
``$g$-admissible lifting''
$M\in \Pic(X)_\QQ$ of $\xi$, and a unique $g$-admissible adelic line bundle $\overline M_g$ extending $M$. Here we a priori do not have $L=M$. But we can still apply the arithmetic Hodge index theorem by the above argument to force 
$\overline L_f=\overline M_g$.

\subsubsection*{Acknowlegement}

The series grows out of the authors' attempt to understand a
counter-example by Dragos Ghioca and Thomas Tucker on the previous
dynamical Manin--Mumford conjecture. The results have been reported
in ``International Conference on Complex Analysis and Related
Topics" held in Chinese Academy of Sciences on August
20-23, 2009.

The authors would like to thank Xander Faber, Dragos Ghioca, Walter Gubler,
Yunping Jiang,
Barry Mazur, Thomas Tucker, Yuefei Wang, Chenyang Xu, Shing-Tung Yau, Yuan Yuan, Zhiwei Yun
and Wei Zhang for many helpful discussions during the preparation of
this paper.

The first author was supported by the Clay Research Fellowship and is supported by a grant from the NSF of the USA, and the
second author is supported by a grant from the NSF of the USA and a grant from
Chinese Academy of Sciences.

\section{Arithmetic Hodge index theorem}

The goal of this section is to prove Theorem \ref{hodge}. 
After introducing some basic positivity notions, 
we prove the theorem by a few steps which are clear from the section titles. 

\subsection{Terminology on arithmetic positivity}

Let $K$ be a number field. By an \emph{arithmetic variety} $\CX$ over $O_K$, we mean an integral scheme $\CX$, projective and flat over $O_K$. 

Recall that a \emph{Hermitian line bundle} on $\CX$ is a pair $\CLL=(\CL,\|\cdot\|)$ consisting of a line bundle $\CL$ on $\CX$ and a metric $\|\cdot\|$ on $\CL(\CC)$ invariant under the action of the complex conjugation. The metric is assumed to satisfy the regularity that, for any analytic map from any complex open ball $B$ (of any dimension) to $\CX(\CC)$, the pull-back of $\|\cdot\|$ gives a smooth Hermitian metric on the pull-back of $\CL(\CC)$ to $B$. We say that the metric is \emph{semipositive} if any such pull-back has a positive semi-definite Chern form. 

\begin{defn} \label{def-positivity2} 
\begin{enumerate} [(1)] 
\item We say that a Hermitian line bundle $\CLL=(\CL,\|\cdot\|)$ on $\CX$ is \emph{nef} if the Hermitian metric $\|\cdot\|$ is semipositive, and the intersection $(\CLL|_\CY)^{\dim \CY}\geq 0$ for any closed subvariety $\CY$ of $\CX$.
\item We say that a Hermitian line bundle $\CLL=(\CL,\|\cdot\|)$ on $\CX$ is \emph{ample} if the generic fiber $\CL_K$ is ample on $X$, $\CLL$ is nef, and the intersection $(\CLL|_\CY)^{\dim \CY}> 0$ for any horizontal closed subvariety $\CY$ of $\CX$.
\item We say that a Hermitian line bundle $\CLL=(\CL,\|\cdot\|)$ on $\CX$ is \emph{vertical} if the generic fiber $\CL_K$ is trivial on $X$. 
\end{enumerate}
\end{defn}
Denote by $\wh\Pic(\CX)$ the group of Hermitian bundles on $\CX$. 
Denote by $\wh\Pic(\CX)_{\nef}$ (resp. $\wh\Pic(\CX)_{\amp}$, $\wh\Pic(\CX)_{\vert}$) the set of nef (resp. ample, vertical) Hermitian line bundles on $\CX$.

Let $X$ be a projective variety over $K$. By an \emph{integral model $\CX$ of $X$ over $O_K$}, we mean an arithmetic variety $\CX$ over $O_K$ with a fixed isomorphism $\CX_K=X$. 
Let $L$ be a line bundle on $X$. If $\CLL$ is a Hermitian line bundle on $\CX$ with generic fiber $\CL_K=L$, then we say that $\CLL$ is an \emph{integral model} of $L$ on $\CX$. We also say that $(\CX,\CLL)$ is an \emph{integral model} of $(X,L)$ over $O_K$. 

An \emph{adelic line bundle on $X$} is a pair $\overline L=(L,\{\|\cdot\|_v\}_v)$ consisting of a line bundle $L$ on $X$ and a collection of $K_v$-metrics $\|\cdot\|_v$ of $L(\overline K_v)$ on $X(\overline K_v)$ over all places $v$ of $K$. The $K_v$-metrics is assumed to be continuous and Galois invariant, and the collection is assumed to be coherent in that there is an integral model of $(X,L)$ inducing a collection of $K_v$-metrics on $L(\overline K_v)$ which agrees with $\{\|\cdot\|_v\}_v$ for all but finitely many places $v$. 
We refer to \cite{Zh2} for more details on the definition. 
The following is a copy of Definition \ref{def-positivity}, except that (5) is new.  
\begin{defn} 
Let $\overline L, \overline M$ be adelic line bundles on $X$.
We make the following definitions. 
\begin{enumerate} [(1)]
\item We say that $\overline L$ is \emph{nef} if it the adelic metric of $L$ is a uniform limit of metrics induced by nef Hermitian line bundles on integral models of $X$. 
\item We say that $\overline L$ is \emph{integrable} if it is the difference of two nef adelic line bundles on $X$.  
\item We say that $\overline L$ is \emph{ample} if
$L$ is ample, $\overline L$ is nef, and $(\overline L|_Y)^{\dim Y+1}> 0$ for any closed subvariety $Y$ of $X$.
\item We say that $\overline M$ is \emph{$\overline L$-bounded} if there is an integer $m >0$ such that both $m\overline L+ \overline M$ and $m\overline L- \overline M$ are nef. 
\item We say that $\overline L$ is \emph{vertical} if $L$ is trivial on $X$.
\end{enumerate}
\end{defn}
Denote by $\wh\Pic(X)$ the group of adelic line bundles on $X$. 
By the above definitions, $\wh\Pic(X)$ has subsets $\wh\Pic(X)_{\nef}$, $\wh\Pic(X)_{\amp}$, $\wh\Pic(X)_{\intb}$ and $\wh\Pic(X)_{\vert}$. 

All the positivity notions can be extended to the concepts of $\QQ$-line bundles and $\RR$-line bundles. For example, the space of integrable $\QQ$-line bundles form
a vector space $\wh\Pic(X)_{\intb}\otimes_\ZZ\QQ$. 

In the end, we present a lemma which shows the strength of the arithmetic ampleness defined above. 
\begin{lem} \label{ample}
Let $\overline L$ be an adelic line bundle on $X$. 
Then the following are equivalent:
\begin{enumerate}[(1)]
\item $\overline L$ is ample, i.e., $L$ is ample, $\overline L$ is nef, and $(\overline L|_Y)^{\dim Y+1}> 0$ for any closed subvariety $Y$ of $X$.
\item $L$ is ample, and $\overline L-\pi^*\overline N$ is nef for some adelic line bundle $\overline N$ on $\Spec K$ with $\wh\deg(\overline N)>0$.
\end{enumerate}
\end{lem}

\begin{proof}
It is easy to see that (2) implies (1). 
The implication from (1) to (2) follows the idea of the successive minima of \cite{Zh2}. 
Assume (1) in the following. We claim that the absolute minima
$$
\lambda(\overline L)=\inf_{x\in X(\overline K)} h_{\overline L}(x) 
$$ 
is strictly positive. 
Apply the arithmetic Hilbert--Samuel formula for adelic line bundles in \cite[Theorem 1.7]{Zh3}, which extends the original formula of Gillet--Soul\'e \cite{GS2}. Combine it with the adelic Minkowski theorem. 
Replacing $\overline L$ by a multiple if necessary, we can find a nonzero section $s\in H^0(X,L)$ such that $\|s\|_{\BA} < 1$. 
Here $$
\|s\|_{\BA}=\prod_{v} \|s\|_{v,\sup}, \quad
\|s\|_{v,\sup}=\sup_{x\in X(\overline K_v)} \|s(x)\|_{v}.
$$
For any closed point $x\in X(\overline K)$ such that $s(x)\neq 0$, 
use $s$ to compute the height of $x$. It gives
$$
\wh\deg(\overline L|_x) \geq -\log \|s\|_{\BA}>0. 
$$
To show $\lambda>0$, it suffices to consider algebraic points of the support of $\div(s)$. It has a smaller dimension, and can be done by induction. 

Once we have $\lambda>0$, the proof is immediate. In fact, we simply have
$$
\lambda(\overline L-\pi^*\overline N)=\lambda(\overline L)-\wh\deg(\overline N) > 0 
$$ 
for any adelic line bundle $\overline N$ on $\Spec K$ with 
$\wh\deg(\overline N)<\lambda$. 
In this case, $\overline L-\pi^*\overline N$ is nef.  
\end{proof}

Under some regularity assumption on the metric at infinity, the arithmetic Nakai--Moishezon criterion of Zhang \cite{Zh3} holds for an ample adelic bundle $\overline L$. In particular, the regularity assumption is satisfied automatically if $X$ is smooth over $K$.

\subsection{Vertical case}

We prove Theorem \ref{hodge} assuming that $M$ is trivial on $X$.

\subsubsection*{The inequality} 

To apply the full strength of intersection theory, we need some regularity
property on arithmetic varieties. We say that an arithmetic variety is \emph{vertically factorial} 
if all irreducible components of its special fibers are Cartier divisors. 
Let $\CX$ be any arithmetic variety. After blowing up the irreducible components
of the special fibers of $\CX$ which are not Carter divisors successively (in any order), we end up with a vertically factorial arithmetic variety
$\CX'$ which dominates $\CX$.

\begin{prop}\label{hodge2}
Let $\pi:\CX\to\Spec O_K$ be an arithmetic variety of relative dimension $n$.
Then the following are true:
\begin{itemize}
\item[(1)] If $\overline \CM\in \wh\Pic(\CX)_{\vert}$ and $\overline\CL_1, \cdots, \overline\CL_{n-1}\in \wh\Pic(\CX)_{\nef}$,
then
$$
\overline \CM^2\cdot \CLL_1\cdot \CLL_2 \cdots \CLL_{n-1}\leq 0.
$$

\item[(2)] If $\CMM_1,\CMM_2\in \wh\Pic(\CX)_{\vert}$ and $\overline\CL_1, \cdots, \overline\CL_{n-1}\in \wh\Pic(\CX)_{\nef}$, then
$$
(\CMM_1\cdot \CMM_2 \cdot \CLL_1\cdot \CLL_2 \cdots \CLL_{n-1})^2\leq
(\CMM_1^2 \cdot \CLL_1\cdot \CLL_2 \cdots \CLL_{n-1})(\CMM_2^2 \cdot \CLL_1\cdot \CLL_2\cdots \CLL_{n-1}).
$$

\item[(3)] Assume that $\CX$ is vertically factorial. 
Let $\overline \CM\in \wh\Pic(\CX)_{\vert}$ and $\CLL_1, \cdots, \CLL_{n-1}\in \wh\Pic(\CX)_{\nef}$. Assume that $\CL_i$ is ample on $\CX$ and that the Hermitian metric of $\CLL_i$ is strictly positive 
on the smooth locus of $\CX(\CC)$ for every $i$. Then
$$
\overline \CM^2\cdot \CLL_1\cdot \CLL_2 \cdots \CLL_{n-1}= 0\
\Longleftrightarrow\ \CMM\in \pi^* \wh\Pic(O_K).
$$

\end{itemize}

\end{prop}

\begin{proof}

The results are well-known for $n=1$ and more or less for general $n$.
The proof for general $n$ is not more difficult than the case $n=1$. We include
it here for convenience.

Note that (2) is the Cauchy--Schwartz inequality induced by (1).
In fact, let $x, y\in \BR$ be variables. The quadratic form
\begin{eqnarray*}
&&(x\CMM_1+y\CMM_2)^2\\
&=&  x^2\ \CMM_1^2\cdot \CLL_1 \cdots \CLL_{n-1}
+2 xy \ \CMM_1\cdot \CMM_2\cdot \CLL_1 \cdots \CLL_{n-1}
+y^2\  \CMM_2^2\cdot \CLL_1 \cdots \CLL_{n-1}
\end{eqnarray*}
is negative semi-definite.
It follows that the discriminant is negative or zero, which gives the
inequality.

To prove (1), we can assume that $\CX$ is also generically smooth and vertically factorial. In fact, let $\CX'$ be the successive blowing-up of the non-Cartier irreducible components of the special fiber of the generic desingularization of $\CX$ as above. Replace $\CX$ by $\CX'$ and all the Hermitian line bundles by the pull-backs on $\CX'$. 

Let $\overline D=(D,g)$ be a vertical arithmetic divisor representing $\CMM$. 
Here the Green's function $g$ is a continuous function on $\CX(\CC)$. 
By definition,
$$
\CM^2\cdot \CLL_1 \cdots \CLL_{n-1}=
\sum_{v\nmid \infty} D_v^2\cdot \CL_1 \cdots \CL_{n-1}
+\int_{\CX(\CC)}g \frac{\partial\overline\partial}{\pi i} g
\wedge c_1(\CLL_1)\cdots c_1(\CLL_{n-1}).
$$
Here $D_v$ denotes the part of $D$ above $v$. 
By integration by parts, the integral becomes
$$
-\frac{1}{\pi i} \int_{\CX(\CC)} \partial g\wedge\overline\partial  g
\wedge c_1(\CLL_1)\cdots c_1(\CLL_{n-1})\leq 0.
$$

To show $D_v^2\cdot \CL_1 \cdots \CL_{n-1}\leq 0$,
enumerate the irreducible components of the special fiber $\CX_v$ above $v$ by $V_1, \cdots, V_r$.
They are Cartier divisors by our assumption.
We have $\CX_v=\sum_{i=1}^r a_iV_i$ with multiplicity $a_i>0$.
For convenience, denote $E_i=a_iV_i$.
Write $D_v=\sum_{i=1}^r b_iE_i$ with some $b_i\in\BR$. We have
$$
D_v^2\cdot \CL_1 \cdots \CL_{n-1}
=\sum_{i,j=1}^r b_ib_jE_i\cdot E_j\cdot \CL_1 \cdots \CL_{n-1}.
$$
Note that
$$
\sum_{j=1}^r b_i^2\ E_i\cdot E_j\cdot \CL_1 \cdots \CL_{n-1}=0,
\quad \forall i.
$$
We obtain
\begin{eqnarray*}
D_v^2\cdot \CL_1 \cdots \CL_{n-1}
&=&- \frac 12 \sum_{i,j=1}^r (b_i-b_j)^2 E_i\cdot E_j\cdot \CL_1 \cdots
\CL_{n-1}\\
&=&- \frac 12 \sum_{i\neq j} (b_i-b_j)^2 E_i\cdot E_j\cdot \CL_1 \cdots
\CL_{n-1}\\
&\leq& 0.
\end{eqnarray*}
It finishes (1).

As for (3), we look at the conditions for the equality at every place.  
The integrals equal zero if and only if $g$ is constant on every connected component of $\CX(\CC)$. 
As for the intersection number above $v$, we have
$$
E_i\cdot E_j\cdot \CL_1 \cdots \CL_{n-1}>0, \qquad i\ne j
$$
as long as $V_i\cap V_j$ is nonempty.
For such $i,j$, we have $b_i=b_j$.
It gives the equality of all $b_i$ since the whole special fiber $\CX_v$ is connected.
\end{proof}

By taking limit, we have the following consequence. It gives the inequality of Theorem \ref{hodge} in the vertical case. 

\begin{prop} \label{hodge3}
Let $\pi:X\to \Spec K$ be a projective variety of dimension $n$.
Then the following are true:
\begin{itemize}
\item[(1)] If $\overline M\in \wh\Pic(X)_{\vert}$ and $\overline L_1, \cdots, \overline L_{n-1}\in \wh\Pic(X)_{\nef}$,
then
$$
\overline M^2\cdot \overline L_1\cdot \overline L_2 \cdots \overline L_{n-1}\leq 0.
$$

\item[(2)] If $\overline M_1,\overline M_2\in \wh\Pic(X)_{\vert}$ and $\overline\CL_1, \cdots, \overline\CL_{n-1}\in \wh\Pic(X)_{\nef}$, then
$$
(\overline M_1\cdot \overline M_2 \cdot \overline L_1\cdot \overline L_2 \cdots \overline L_{n-1})^2\leq
(\overline M_1^2 \cdot \overline L_1\cdot \overline L_2 \cdots \overline L_{n-1})(\overline M_2^2 \cdot \overline L_1\cdot \overline L_2\cdots \overline L_{n-1}).
$$
\end{itemize}

\end{prop}

\subsubsection*{Variational method}

The key to obtain the equality part of Theorem \ref{hodge} is the following result, which was inspired by \cite{Bl}. 

\begin{lem}\label{variation}
Let $\overline M, \overline L_1, \cdots, \overline L_{n-1}$ be integrable adelic line bundles on $X$ such that the following conditions hold:
\begin{enumerate}[(1)]
\item $M$ is trivial on $X$;
\item $\overline M$ is $\overline L_i$-bounded for every $i$;
\item $\overline M^2\cdot \overline L_1\cdots \overline L_{n-1}=0$.
\end{enumerate}
Then for any nef adelic line bundles $\overline L_i^0$ on $X$ with underlying bundle $L_i^0=L_i$, and any integrable adelic line bundle $\overline M'$ with trivial underlying line bundle $M'$, we have
$$ \overline M\cdot \overline M'\cdot \overline L_1^0\cdots \overline L_{n-1}^0=0.$$
\end{lem}

\begin{proof} 
Note that (2) implies every $\overline L_i$ is nef. 
Replacing with $\overline L_i$ by a large multiple if necessary, we can assume that both $\overline L_i\pm \overline M$ are nef for each $i$. Denote $\overline L_{i}^\pm =\overline L_i\pm \overline M$ in the following.

First, we have the equality
$$\overline M^2\cdot \overline L_1^{\epsilon(1)}\cdots\overline L_{n-1}^{\epsilon (n-1)}=0$$
for any sign function $\epsilon: \{1, \cdots, n-1\}\to \{+, -\}$. 
In fact, we can find a constant $t>0$ such that $L_i-t L_i^{\epsilon(i)}$ is nef for any $i$. Then Proposition \ref{hodge3} gives
$$
0=\overline M^2\cdot \overline L_1\cdots \overline L_{n-1} \leq 
\overline M^2\cdot t\overline L_1^{\epsilon(1)}\cdots t\overline L_{n-1}^{\epsilon (n-1)} \leq 0.
$$ 
It forces
$$\overline M^2\cdot \overline L_1^{\epsilon(1)}\cdots\overline L_{n-1}^{\epsilon (n-1)}=0.$$

Second, we claim that
$$
\overline M^2\cdot \overline L_1^0\cdots \overline L_{k-1}^0\cdot \overline L_{k}^{\epsilon (k)}\cdots\overline L_{n-1}^{\epsilon (n-1)}=0, \quad
\forall\ k=1,\cdots,n.
$$
When $k=n$, then we can apply Proposition \ref{hodge3} (2) to conclude the proof.

Prove the claim by induction. 
We already have the case $k=1$. Assume the equality for a general $k$ for all sign functions $\epsilon$. 
Apply Proposition \ref{hodge3} (2) to the vertical adelic line bundles $\overline M$ and $\overline M_{k}^{\epsilon (k)}=\overline L_k^{\epsilon (k)}-\overline L_k^0$, we have 
$$
\overline M\cdot \overline M_{k}^{\epsilon (k)} \cdot \overline L_1^0\cdots \overline L_{k-1}^0\cdot \overline L_{k}^{\epsilon (k)}\cdots\overline L_{n-1}^{\epsilon (n-1)}=0.
$$
Replace $\overline L_{k}^{\epsilon (k)}$ by $\overline L_{k}^+$ and $\overline L_{k}^-$, and take the difference. We have
$$
\overline M\cdot \overline M_{k}^{\epsilon (k)} \cdot \overline L_1^0\cdots \overline L_{k-1}^0\cdot \overline M\cdot \overline L_{k+1}^{\epsilon (k+1)} \cdots\overline L_{n-1}^{\epsilon (n-1)}=0.
$$
It is just
$$
\overline M^2 \cdot \overline L_1^0\cdots \overline L_{k-1}^0\cdot (\overline L_k^{\epsilon (k)}-\overline L_k^0)\cdot \overline L_{k+1}^{\epsilon (k+1)} \cdots\overline L_{n-1}^{\epsilon (n-1)}=0.
$$
The left-hand side splits to a difference of two terms. One term is zero by the induction assumption for $k$. It follows that the other term is also zero, which gives $$
\overline M^2 \cdot \overline L_1^0\cdots \overline L_{k-1}^0\cdot \overline L_k^0\cdot \overline L_{k+1}^{\epsilon (k+1)} \cdots\overline L_{n-1}^{\epsilon (n-1)}=0.
$$
It is exactly the case $k+1$. 
The proof is complete.
\end{proof}

\subsubsection*{Condition of equality}

To reduce adelic metrics to model metrics, we first introduce the ``push-forward'' of adelic line bundle to integral models. 

Let $\overline M$ be an integrable adelic line bundle on $\CX$ with $M=\CO_X$ trivial. For any place $v$ of $K$, the function $\log\|1\|_v$ on $X(\overline K_v)$ coming from the metric of $\overline M$ extends to a continuous function on the analytic space $X_{K_v}^\an$. By definition, $\log\|1\|_v=0$ for all but finitely many $v$.

Let $\CX$ be a vertically factorial integral model of $X$ over $O_K$.  
Define a vertical arithmetic $\RR$-divisor $\overline D_{\CX}$ on $\CX$ by
$$
\overline D_{\CX}:= (-\sum_{(V,v)} \log\|1\|_v(\eta_V), \ -\log\|1\|_\infty).
$$
Here the summation is over all pairs $(V,v)$, where $v$ is a non-archimedean place of $K$, and $V$ is an irreducible component of the fiber $\CX_v$ of $\CX$ above $v$. 
The point $\eta_V\in X_{K_v}^\an$ denotes the Shilov point corresponding to $V$, which is the unique preimage of the generic point of $V$ under the reduction map $X_{K_v}^\an\to \CX_v$. 

To get a line bundle, we denote by $\overline M_\CX\in \wh\Pic(\CX)_\RR$ the Hermitian line bundle on $\CX$ associated to $\overline D_\CX$. 
We have the following basic result. 

\begin{lem} \label{pushforward}
\begin{enumerate}[(1)]
\item For any integrable Hermitian line bundles $\CLL_1, \cdots, \CLL_{n}$ on $\CX$, we have
$$\overline M\cdot \CLL_1\cdots \CLL_{n}=\overline M_\CX\cdot \CLL_1\cdots \CLL_{n}.$$
\item There is a sequence $\{\CX_m\}_m$ of vertically factorial integral model of $X$ over $O_K$ such that the adelic metric induced by the push-forward 
$\overline M_{\CX_m}$ converges uniformly to the adelic metric of $\overline M$. 
\end{enumerate}
\end{lem}

\begin{proof} 
By definition, there is a sequence of vertically factorial integral models 
$(\CX_m, \CMM_m)$ of $(X,M)$ over $O_K$, which induces a sequence of adelic line bundles $\overline M_m$ convergent to $\overline M$. 
Here by abuse of terminology, convergence means uniform convergence of induced adelic metrics. 

For (2), we prove that $\overline M_{\CX_m}$ converges to $\overline M$. It suffices to show that $\overline N_m=\overline M_{\CX_m}-\overline M_m$ converges to the trivial adelic line bundle. Let $\overline E=(E, g_E)$ be the vertical arithmetic divisor associated to the section 1 of $\overline N_m$. By definition, the function $g_E$ on $X(\CC)$ and the multiplicities of irreducible components of special fibers of 
$\CX$ in $E$ converges uniformly to 0. It follows that $\overline N_m$ 
converges to the trivial adelic line bundle. 

For (1), we can further assume that each $\CX_m$ dominates $\CX$ by a birational morphism $\alpha_m:\CX_m\to \CX$. 
By definition, $\overline M_{m,\CX}=(\alpha_m)_* \CMM_m$ converges to $\overline M_{\CX}$, which can be understood as the convergence of the corresponding vertical divisors.
Then (1) is just the limit of the projection formula
$$
\overline M_m\cdot \CLL_1\cdots \CLL_{n}=\overline M_{m,\CX}\cdot \CLL_1\cdots \CLL_{n}.
$$
\end{proof}

Now we are ready to prove the equality part of Theorem \ref{hodge} in the vertical case.
Let $\overline M$ and $\overline L_1,\cdots \overline L_{n-1}$ be as in the equality part of theorem. Assume that $M$ is trivial. 

Let $\CX$ be a vertically factorial integral model of $X$. 
For each $m$, let $\CLL_1, \cdots, \CLL_{n-1}\in \wh\Pic(\CX)_{\nef}$
such that $\CL_i$ is ample on $\CX$ and that the Hermitian metric of $\CLL_i$ is strictly positive 
on the smooth locus of $\CX(\CC)$
for every $i$. By Lemma \ref{variation}, 
$$
\overline M \cdot \overline M_\CX\cdot \CLL_1\cdots \CLL_{n-1}=0.
$$
By Lemma \ref{pushforward} (1), it becomes
$$
\overline M_\CX\cdot \overline M_\CX\cdot \CLL_1\cdots \CLL_{n-1}=0.
$$
By Proposition \ref{hodge2} (3), $\overline M_\CX\in \pi^* \wh\Pic(O_K)$ where $\pi:\CX\to \Spec O_K$ denotes the structure morphism by abuse of notation. 

By Lemma \ref{pushforward} (2), the metric of $\overline M$ can be approximated uniformly by the metric induced by $\overline M_\CX$ when varying $\CX$. It follows that the limit $\overline M$ actually lies in $\pi^* \wh\Pic(K)$. 
It finishes the proof of Theorem \ref{hodge}.

\subsection{Case of curves}

When $\dim X=1$, Theorem \ref{hodge} can be easily reduced to the result of Faltings \cite{Fal} and Hriljac \cite{Hr} on integral models.
Note that $\overline L_i$ does not show up in the theorem. 

We can assume that $X$ is smooth by a desingularization. 
Take any integral model $\CX$ of $X$ over $O_K$.
Then we can find some $\CMM\in \wh\Pic(\CX)_\QQ$ extending $M$ whose intersection with any vertical arithmetic divisor on $\CX$ is 0.
Denote by $\overline M_0\in \wh\Pic(X)$ the adelic line bundle induced by 
$\CMM$.
Define a vertical adelic line bundle $\overline N$ by
$$\overline M=\overline M_0+\overline N. $$
By definition, $\overline M_0\cdot\overline N=0$ since $\overline M_0$ is perpendicular to all vertical classes.
It follows that 
$$\overline M^2=\overline M_0^2+\overline N^2\leq 0. $$
Here $\overline M_0^2= -2\ \wh h(M)\leq 0$ by the positivity of the Neron--Tate height, and $\overline N^2\leq 0$ follows from the vertical case we treated before. 
The equality is attained if and only if $\overline M$ is torsion in 
$\wh\Pic(X)/\pi^*\wh\Pic(K)$ and
and $\overline N\in \pi^*\wh\Pic(K)$. The result is proved.

\subsection{Inequality in the general case}

Now we prove the inequality of Theorem \ref{hodge} in the general case by induction on $n=\dim X$. We have already treated the case $n=1$, so we assume $n\geq 2$ in the following.

\subsubsection*{Reducing to the model case}

Consider the inequality part of the theorem. 
Recall $\overline M\in \wh\Pic(X)_\intb$ and $\overline L_1, \cdots, \overline L_{n-1}\in \wh\Pic(X)_\nef$. The theorem assumes $M\cdot L_1\cdots L_{n-1}=0$ and that each $L_i$ is big on $X$. We are going to prove 
$$\overline M^2\cdot \overline L_1\cdots \overline L_{n-1} \le 0.$$
By a resolution of singularities, we can assume that $X$ is smooth. 
Furthermore, we claim that we can further assume that each $\overline L_i$ is ample.

In fact, fix an ample adelic line bundle $\overline A$ on $X$. 
Take a small rational number $\epsilon>0$. 
Set $\overline L_i'= \overline L_i+\epsilon \overline A$
and $\overline M'=\overline M+\delta \overline A$. 
Here $\delta$ is a number such that 
$$
M'\cdot L_1'\cdots L_{n-1}'=(M+\delta A)\cdot L_1'\cdots L_{n-1}'=0.
$$
It determines
$$\delta= -\frac{M\cdot L_1'\cdots L_{n-1}'}{A\cdot L_1'\cdots L_{n-1}'}.$$
As $\epsilon\to 0$, we have $\delta\to 0$ since
$$M\cdot L_1'\cdots L_{n-1}'\to M\cdot L_1\cdots L_{n-1}=0,$$
$$A\cdot L_1'\cdots L_{n-1}'\to A\cdot L_1\cdots L_{n-1}>0.$$
The last inequality uses the assumption that $L_i$ is 
big and nef for each $i$.
Therefore, the inequality $\overline M^2\cdot\overline L_1\cdots \overline L_{n-1}\leq 0$
is the limit of the inequality $\overline M'^2\cdot\overline L_1'\cdots \overline L_{n-1}'\leq 0$. Here every $\overline L_i'$ is ample. So the claim is achieved.

Go back to the inequality 
$$\overline M^2\cdot \overline L_1\cdots \overline L_{n-1} \le 0.$$
We have the condition
$M\cdot L_1\cdots L_{n-1}=0$.
We further assume that $X$ is smooth, and $\overline L_i$ is ample for $i=1,\cdots, n-1$.
By approximation, it suffices to prove
$$
\CMM^2\cdot\CLL_1\cdots \CLL_{n-1}\leq 0
$$
under the following assumptions:
\begin{itemize}
\item $\CX$ is a normal integral model of $X$ over $O_K$;
\item $\CMM, \CLL_1, \cdots, \CLL_{n-1}$ are Hermitian line bundles on 
$\CX$ with smooth Hermitian metrics and with generic fiber $M, L_1,\cdots, L_{n-1}$;
\item $\CLL_i$ ample on $\CX$ with strictly positive metric on $\CX(\CC)$ for each $i=1,\cdots, n-1$.
\end{itemize}
It was proved by Moriwaki \cite{Mo1} in the case that all $\CLL_i$ are equal. 
The current case is similar, but we still sketch it in the following.

\subsubsection*{The model case} 
Here we prove the inequality in the above model case. 

First, by $M\cdot L_1\cdots L_{n-1}=0$, there is a metric 
$\|\cdot\|_0$ on $\CM$, unique up to scalars, such that the Chern form of $\CMM'=(\CM,\|\cdot\|_0)$ gives
$$c_1(\CMM') c_1(\CLL_{1}) \cdots c_1(\CLL_{n-1})
=0 $$
pointwise on $X(\CC)$. 
It is given by an elliptic equation. 
We refer to \cite[Corollary 2.2 A2]{Gr} for a proof. 
Write the original metric of $\CMM$ as
$e^{-\phi}\|\cdot\|_0$ for a real-valued smooth function 
$\phi$ on $X(\CC)$.
Then we have 
$$
\CMM^2\cdot\CLL_1\cdots \CLL_{n-1}
=\CMM'^2\cdot\CLL_1\cdots \CLL_{n-1}
-\int_{X(\CC)}\phi \frac{1}{\pi i}\partial\overline\partial \phi 
\wedge  c_1(\CLL_{1}) \cdots c_1(\CLL_{n-1}).
$$
By integration by parts, the second term on the right 
$$
-\int_{X(\CC)}\phi \frac{1}{\pi i}\partial\overline\partial \phi 
\wedge  c_1(\CLL_{1}) \cdots c_1(\CLL_{n-1})
= \int_{X(\CC)} \frac{1}{\pi i}\partial\phi\wedge \overline\partial \phi 
\wedge  c_1(\CLL_{1}) \cdots c_1(\CLL_{n-1}).
$$
It is non-positive since the volume from
$$
\frac{1}{\pi i}\partial\phi\wedge \overline\partial \phi 
\wedge  c_1(\CLL_{1}) \cdots c_1(\CLL_{n-1})\leq 0.
$$
Hence, it suffices to prove
$$
\CMM'^2\cdot\CLL_1\cdots \CLL_{n-1}\leq 0.
$$

By Moriwaki's arithmetic Bertini theorem in \cite{Mo1}, replacing $\CLL_{n-1}$ by a tensor power if necessary, there is a nonzero section $s\in H^0(\CX,\CL_{n-1})$ satisfying the following conditions:
\begin{itemize}
\item The supremum norm $\|s\|_{\sup}=\sup_{x\in X(\CC)} \|s(x)\|<1$;
\item The horizontal part of $\div(s)$ on $\CX$ is a generically smooth arithmetic variety $\CY$;
\item The vertical part of $\div(s)$ on $\CX$ is a linear combination $\sum_{\wp}a_\wp \CX_\wp$ of smooth fibers $\CX_\wp$ of $\CX$ above (good) prime ideals $\wp$ of $O_K$.
\end{itemize}
Then by the intersection formula,
\begin{eqnarray*}
&&\CMM^2\cdot\CLL_1\cdots \CLL_{n-2}\cdot \CLL_{n-1} \\
&=&\CMM^2\cdot\CLL_1\cdots \CLL_{n-2}\cdot \CY
+\sum_{\wp} a_\wp\ \CMM^2\cdot\CLL_1\cdots \CLL_{n-2}\cdot \CX_\wp  \\
&& -\int_{X(\CC)} \log\| s\| c_1(\CMM)^2 c_1(\CLL_{1}) \cdots c_1(\CLL_{n-2}).
\end{eqnarray*}

It suffices to prove each term on the right-hand side is non-positive. 
The ``main term'' 
$$\CMM^2\cdot\CLL_1\cdots \CLL_{n-2}\cdot \CY\leq 0$$
by induction hypothesis.
The vertical part
$$
\CMM^2\cdot\CLL_1\cdots \CLL_{n-2}\cdot \CX_\wp
=M^2\cdot L_1\cdots L_{n-2}\log N_\wp \leq 0
$$
follows from the geometric Hodge index theorem on the algebraic variety
$X$ (or $\CX_\wp$). 
It remains to check 
$$
-\int_{X(\CC)} \log\| s\| c_1(\CMM)^2 c_1(\CLL_{1}) \cdots c_1(\CLL_{n-2})
\leq 0.
$$
In fact, we have a pointwise inequality
$$c_1(\CMM)^2 c_1(\CLL_{1}) \cdots c_1(\CLL_{n-2})
\leq 0.$$
See \cite[Lemma 2.1A]{Gr}.
It is called Aleksandrov's lemma, since it was essentially due to Aleksandrov \cite{Al} assuming the existence of the metric $\|\cdot\|_0$.

\subsection{Equality in the general case}

Now we prove the second part of Theorem \ref{hodge} in the general case. We have already treated the case $n=1$, so we assume $n\geq 2$ in the following.  

\subsubsection*{Argument on the generic fiber}

Assume the conditions in the equality part of the Theorem \ref{hodge}, which particularly includes
$$\overline M^2\cdot \overline L_1\cdots \overline L_{n-1} = 0.$$
We first show that $M$ is numerically trivial on $X$ by the condition that $\overline L_{n-1}$ is ample. 

By Lemma \ref{ample}, $\overline L_{n-1} '=\overline L_{n-1} - \pi^*\overline N$ is nef for some $\overline N\in \wh\Pic(K)$ with $c=\wh\deg(\overline N)>0$.
Then 
$$\overline M^2\cdot \overline L_1 \cdots \overline L_{n-2}\cdot \overline L_{n-1}
=\overline M^2\cdot \overline L_1 \cdots \overline L_{n-2}\cdot \overline L_{n-1}'
+c \ M^2\cdot  L_1 \cdots  L_{n-2}.$$
Applying the inequality of the theorem to $(\overline M, \overline L_1, \cdots, \overline L_{n-2}, \overline L_{n-1}')$, we have
$$\overline M^2\cdot \overline L_1 \cdots \overline L_{n-2}\cdot \overline L_{n-1}' \leq 0.$$
By the Hodge index theorem on $X$ in the geometric case, we have
$$M^2\cdot  L_1 \cdots  L_{n-2}\leq 0.$$
Hence,
$$\overline M^2\cdot \overline L_1 \cdots \overline L_{n-2}\cdot \overline L_{n-1}'
=\ M^2\cdot  L_1 \cdots  L_{n-2}
=0.$$

On the variety $X$, we have
$$
M\cdot  L_1 \cdots  L_{n-2}\cdot L_{n-1}=0,
\quad M^2\cdot  L_1 \cdots  L_{n-2}=0.
$$
By the Hodge index theorem on algebraic varieties, we conclude that $M$
is numerically trivial. 
See Theorem \ref{lefschetz2} in the appendix.

\subsubsection*{Numerically trivial case}

We have proved that $M$ is numerically trivial on $X$. 
Here we continue to prove that $M$ is a torsion line bundle. 
Then a multiple of $\overline M$ lies in the vertical case, which has already been treated. 

As in the vertical case, the key is still the variational method. 

\begin{lem} \label{variation2}
Let $\overline M, \overline L_1, \cdots, \overline L_{n-1}$ be integrable adelic line bundles on $X$ such that the following conditions hold:
\begin{enumerate}[(1)]
\item $M$ is numeriacally trivial on $X$;
\item $\overline M$ is $\overline L_i$-bounded for every $i$;
\item $\overline M^2\cdot \overline L_1\cdots \overline L_{n-1}=0$.
\end{enumerate}
For any nef adelic line bundles $\overline L_i^0$ on $X$ with underlying bundle $L_i^0$ numerically equivalent to $L_i$, and any integrable adelic line bundle $\overline M'$ with numerically trivial underlying line bundle $M'$, the following are true:
$$ \overline M\cdot \overline M'\cdot \overline L_1^0\cdots \overline L_{n-1}^0=0,$$
$$\overline M^2\cdot \overline M'\cdot \overline L_1^0\cdots \overline L_{n-2}^0 = 0.$$
\end{lem}
\begin{proof}
The first equality can be proved exactly in the same way as in Lemma \ref{variation}. The second equality follows from the first one by vary $\overline L_{n-1}^0$. 
\end{proof}

Go back to the equality part of Theorem \ref{hodge}.  
Apply Bertini's theorem. Replacing $\overline L_{n-1}$ by a positive multiple if necessary, there
is a section $s\in H^0(X,L_{n-1})$ such that $Y=\div(s)$ is a normal subvariety of $X$. 
Then we have 
$$
\overline M^2\cdot \overline L_1\cdots \overline L_{n-2}\cdot \overline L_{n-1}=
\overline M^2\cdot \overline L_1\cdots \overline L_{n-2}\cdot Y.
$$ 
In fact, the difference of two sides is the limit of the intersection of $\overline M^2\cdot \overline L_1\cdots \overline L_{n-2}$ with vertical classes, so it vanishes by the second equality of Lemma \ref{variation2}. 
Hence, 
$$
\overline M^2\cdot \overline L_1\cdots \overline L_{n-2}\cdot Y=0.
$$ 
By Theorem \ref{lefschetz1}, we can assume that 
$\Pic^0(X)_\QQ\to \Pic^0(Y)_\QQ$ is injective. 
Note that some multiple of $M$ lies in $\Pic^0(X)$.
It reduces the problem to $Y$.
The proof is complete since we have already treated the case of curves.

\subsection{Calabi--Yau theorem}

The goal of this section is to treat Theorem \ref{calabiyau} in the non-archimedean case. If $(X,L)$ comes from a number field, it is essentially Theorem \ref{hodge}, the arithmetic Hodge index theorem for vertical adelic line bundles. 

Let $K$ be a non-archimedean field, $X$ be a projective variety over $K$, and $L$ be a line bundle on $X$. Recall that a \emph{$K$-metric} on $L$ is a continuous and $\Gal(\overline K/K)$-invariant collection of
$\overline K$-metric $\|\cdot\|$ on $L(x)$ indexed by $x\in X(\overline K)$.
A \emph{metrized line bundle} $\overline L$ on $X$ is a pair $(L,\|\cdot\|)$ consisting of a line bundle $L$ on $X$ and a $K$-metric $\|\cdot\|$ on $L$.  
Denote by $\wh\Pic(X)$ by the group of all metrized line bundles on $X$.  

The metric is said to be \emph{semipositive} if it is a uniform limit of metrics induced by integral models $(\CX_m,\CL_m)$ of $(X,L)$ over $O_K$
where $\CL_m\in \Pic(\CX)_{\QQ}$ is nef on on fibers of $\CX_m$ above $O_K$. The metric is said to be \emph{integrable} if it is the quotient of two semipositive metrics. By abuse of notations, we also say that the corresponding metrized line bundle are semipositive or integrable if the metric is so. To be compatible with the global case, semipositive metrized line bundles are also called \emph{nef metrized line bundles}. 
Finally, we also have the notion of $\overline L$-bounded as in the global base. 

\begin{thm}[local hodge index theorem]\label{hodge6} 
Let $K$ be a non-archimedean field, and $\pi:X\to \Spec K$ be a geometrically connected projective variety of dimension $n\geq 1$. 
Let $\overline M$ be an integrable metrized line bundle on $X$ with $M$ trivial, and 
$\overline L_1, \cdots, \overline L_{n-1}$ be $n-1$ nef metrized line bundles on $X$.
Then 
$$\overline M^2\cdot \overline L_1\cdots \overline L_{n-1} \le 0.$$
Moreover, if $L_i$ is ample and $\overline M$ is $\overline L_i$-bounded for each $i$, then the equality holds if and only if  $\overline M\in \pi^* \wh\Pic(K)$.
\end{thm} 

The theorem is a local version of Theorem \ref{hodge} in the vertical case. 
If $K$ is the completion of a number field $K_0$ at some place, and $(X,L)$ is the base changes to $K$ of a pair $(X_0,L_0)$ over $K_0$. Then Theorem \ref{hodge6} is equivalent to Theorem \ref{hodge}. 

The proof of Theorem \ref{hodge} in the vertical case is easily translated to a proof of Theorem \ref{hodge6}. We omit it. However, it is worth noting that,  
for a general non-archimedean field $K$, integral models of $X$ over $O_K$ are not Noetherian, so the usual intersection theory is not applicable. In that
case, Gubler \cite{Gu1} introduced an intersection theory using rigid-analytic
geometry, and the translation goes through by his intersection theory. 

Now we go back to Theorem \ref{calabiyau}.
We will prove the following stronger result as in the case of model metrics.

\begin{thm} \label{calabiyauvariant}
Let $L$ be an ample line bundle over $X$, and $\|\cdot\|_1$ and
$\|\cdot\|_2$ be two semipositive metrics on $L$.
View $f=-\log (\|\cdot\|_1/\|\cdot\|_2)$ as a continuous function on
$X^{\mathrm{an}}$.
Then
$$
\int_{X^{\mathrm{an}}} f\ c_1(L, \|\cdot\|_1)^{\dim X}=\int_{X^{\mathrm{an}}} f\
c_1(L, \|\cdot\|_2)^{\dim X}
$$
if and only if $f$ is a constant.
\end{thm}

This theorem is just a simple consequence of Theorem \ref{hodge6}. 
In fact, denote $\overline L_1=(L, \|\cdot\|_1)$ and
$\overline L_2=(L, \|\cdot\|_2)$.
Then the equality of the integrals is just 
$$
(\overline L_1-\overline L_2)\cdot \overline L_1^{n}
=(\overline L_1-\overline L_2)\cdot \overline L_2^{n}.
$$
Here $n=\dim X$.
Equivalently, 
$$
\sum_{i=0}^{n-1} 
(\overline L_1-\overline L_2)^2 \cdot \overline L_1^i\cdot \overline L_2^{n-1-i}=0.
$$
By Theorem \ref{hodge6}, every term in the sum is non-positive. 
It forces
$$
(\overline L_1-\overline L_2)^2 \cdot \overline L_1^i\cdot \overline L_2^{n-1-i}=0, \quad \forall\ i=0,\cdot, n-1.
$$
It follows that
$$
(\overline L_1-\overline L_2)^2 \cdot (\overline L_1+ \overline L_2)^{n-1}=0.
$$
Note that $\overline L_1-\overline L_2$ is vertical and $(\overline L_1+\overline L_2)$-bounded. 
Theorem \ref{hodge6} implies  
$\overline L_1-\overline L_2\in \pi^* \wh\Pic(K)$, which is equivalent to the statement that $f$ is a constant.

\section{Algebraic dynamics}

In this section, we prove Theorem \ref{dynamicsmain}. 
We first introduce a theory of admissible adelic line bundles which will be needed in the proof. Then we prove the theorem using our arithmetic Hodge index theorem.

\subsection{Admissible arithmetic classes}
Let $(X, f, L)$ be a \emph{polarized dynamical system} over a number field $K$, i.e.,
\begin{itemize}
\item $X$ is a projective variety over $K$;
\item $f:X\to X$ is a morphism over $K$;
\item $L\in \Pic(X)_\QQ$ is an ample line bundle such that
$f^*L=qL$ from some $q>1$.
\end{itemize}
By \cite{Zh3}, Tate's limiting argument gives an adelic $\QQ$-line bundle $\overline L_f\in \wh\Pic(X)_{\QQ,\nef}$ extending $L$ and with $f^*\overline L_f=q\overline L_f$.  
In the following we generalize the definition to construct an admissible metric on any line bundle $M\in \Pic(X)$. 

\subsubsection*{Semisimplicity}

By definition, $f^*$ preserves the  exact sequence
$$0\lra \Pic ^0(X)\lra \Pic (X)\lra \NS (X)\lra 0.$$
It is known that $\NS(X)$ is a finitely generated $\ZZ$-module. Assume that $X$ is normal. 
Then $\Pic^0(X)$ is also a finitely generated $\ZZ$-module. 
In fact, the Picard functor $\underline{\Pic}^0_{X/K}$ is represented by an abelian variety $A$. See \cite[Theorem 5.4]{Kl} for example.
Then $\Pic^0(X)=A(K)$ is just the Mordell--Weil group.  
Alternatively, one can obtain the finiteness using only the Picard variety for a resolution of singularities. 

\begin{thm} \label{semisimplicity}
Let $(X, f, L)$ be a {polarized dynamical system} over a number field $K$.
Assume that $X$ is normal. 
\begin{enumerate}[(1)]
\item The operator $f^*$ is semisimple on 
$\Pic^0(X)_\BC$ (resp. $\NS (X)_\BC$) with eigenvalues of absolute values $q^{1/2}$ (resp. $q$). 
\item the operator $f^*$ is semisimple on 
$\Pic(X)_\BC$ with eigenvalues of absolute values $q^{1/2}$ or $q$. 
\end{enumerate} 
\end{thm}
\begin{proof}
The result can be proved using Hodge--Riemann bilinear relation (for Betti cohomology). See Serre \cite{Ser} for the case when $X$ is smooth. Here we present an algebraic proof which can be generalized to positive characteristics. 
It suffices to prove (1), since (2) is a consequence of (1).

We first consider $\NS(X)_\CC$. Write $n=\dim X$ as usual. 
Make the decomposition 
$$\NS(X)_\RR:=\RR L\oplus P(X), \quad
P(X)=\{\xi\in \NS(X)_\RR: \xi\cdot L^{n-1}=0\}.$$
By Theorem \ref{lefschetz2}, the pairing
$$\langle \xi_1,\xi_2 \rangle= \xi_1\cdot \xi_2\cdot L^{n-2}$$
is a negative definite quadratic form on $P(X)$.
The projection formula gives   
$$\langle f^*\xi_1, f^*\xi_2 \rangle=q^2\langle \xi_1, \xi_2 \rangle.$$
If follows that $q^{-1}f^*$ is an orthogonal transformation (with respect to the quadratic form). Then $q^{-1}f^*$ is diagonalizable with eigenvalues of absolute values 1. 

Next we consider $\Pic^0(X)_\BC$. 
For $\xi_1,\xi_2\in \Pic^0(X)$, define a pairing
$$(\xi_1, \xi_2)=\overline \xi_1\cdot\overline \xi_2\cdot \overline L^{n-1}.$$
Here $\overline\xi_i$ is an adelic line bundle on $X$ extending $\xi_i$ and with zero intersection with any vertical classes on $X$ for $i=1,2$, and $\overline L$ is any adelic line bundle on $X$ extending $L$. The intersection does not depend on the choice of the extension $\overline L$ since $\overline\xi_1$ is perpendicular to any vertical class.  

The extension $\overline\xi_i$ always exists. In fact, we already know that
$A=\underline{\Pic}^0_{X/K}$ is an abelian variety. 
Let $P$ be the universal bundle on $X\times A$. Then for any $\alpha\in A(\overline K)$, the line bundle 
$P|_{X\times \alpha}$ on $X$ is exactly the element of $\underline{\Pic}^0_{X/K}(\overline K)$
represented by $\alpha$. 
Rigidify $P$ by $P|_{x_0\times A}=0$ for some point $x_0\in X(K)$. Here we assume $x_0$ exists by extending $K$ if necessary. The multiplication $[2]_X:X\times A\to X\times A$ on the second component gives $[2]_X^*P=2P$. By Tate's limiting argument, we obtain an adelic line bundle $\overline P$ extending $P$ with $[2]_X^*\overline P=2\overline P$. 
Then $\overline P|_{X\times \alpha_i}$ gives the desired extension of $\xi_i$, where $\alpha_i\in A(K)$ is the point representing $\xi_i$.

Go back to the pairing 
$$(\xi_1, \xi_2)=\overline \xi_1\cdot\overline \xi_2\cdot \overline L^{n-1}.$$
It is extended to a pairing on $\Pic^0(X)_\RR$. 
We claim that the pairing is negative definite. 
In fact, let $C$ be a closed non-singular curve in $X$ representing $L^{n-1}$, which exists by Bertini's theorem. Then we have 
$$(\xi_1, \xi_2)=\overline \xi_1|_C\cdot\overline \xi_2|_C
=-2\pair{\xi_1|_C,\ \xi_2|_C}_{\rm NT}.$$
Here we used the Hodge index theorem of \cite{Fal, Hr}. 
By Theorem \ref{lefschetz1}, the map $\Pic^0(X)\to \Pic^0(C)$ has a finite kernel. It follows that the paring is negative definite. 

On the other hand, the projection formula gives   
$$( f^*\xi_1, f^*\xi_2 )=q\ (\xi_1, \xi_2 ).$$
It follows that $q^{-1/2}f^*$ is an orthogonal transformation (with respect to the  pairing). Then $q^{-1/2}f^*$ is diagonalizable with eigenvalues of absolute values 1. The result is proved. 

\end{proof}

By the theorem above, the exact sequence
$$0\lra \Pic ^0(X)_\BC\lra \Pic(X)_\BC\lra \NS (X)_\BC\lra 0.$$
has a splitting
$$\ell_f:  \NS (X)_\BC\lra \Pic(X)_\BC$$
by identifying $\NS (X)_\BC$ with the subspace of $\Pic(X)_\BC$ generated 
by eigenvectors belonging to eigenvalues of absolute values $q$. 
It is easy to see that 
the splitting actually descends to 
$$\ell_f:  \NS (X)_\QQ\lra \Pic(X)_\QQ.$$

\begin{definition}
We say an element of $\Pic(X)_\BC$ is \emph{$f$-pure of weight $1$ (resp. $2$)} if it lies in $\Pic^0(X)_\BC$ (resp. $\ell_f(\NS (X)_\BC)$). 
\end{definition}

\subsubsection*{Admissible metrics}

Now we introduce $f$-admissible metrics for line bundles. Note that the eigenvalues can be imaginary numbers. So we first extend the group of adelic line bundles. 

Recall that $\wh\Pic(X)$ is the group of adelic line bundles on $X$. 
Assume that $\pi:X\to \Spec K$ is geometrically connected.
Write $\BFF$ for $\ZZ, \QQ, \RR$ or $\CC$.
Similar to \cite{Mo4}, we introduce
$$
\wh\Pic(X)_{[\BFF]}:=\frac{\wh\Pic(X)\otimes_\ZZ\BFF}{
(\wh\Pic(X)_\vert\otimes_\ZZ\BFF)_0}. 
$$
Here we describe the subspace in the denominator. 
Denote
$$C(X,\BFF)=\oplus_v C(X_{K_v}^\an,\BFF),$$
where $C(X_{K_v}^\an,\BFF)$ denotes the space of continuous functions from $X_{K_v}^\an$ to $\BFF$. 
When $\BFF=\ZZ,\QQ$, it is endowed with the discrete topology, so $C(X_{K_v}^\an,\BFF)=\BFF$ in these cases. 
The map $\log\|1\|:\wh\Pic(X)_\vert\to C(X,\RR)$ extends to an $\BFF$-linear map
$$\log\|1\|:\wh\Pic(X)_\vert\otimes_\ZZ\BFF\to C(X,\BFF).$$
Define
$$
(\wh\Pic(X)_\vert\otimes_\ZZ\BFF)_0:=\ker(\log\|1\|:\wh\Pic(X)_\vert\otimes_\ZZ\BFF\to C(X,\BFF)).
$$
By definition, it is easy to have
$$
\wh\Pic(X)_{[\ZZ]}=\wh\Pic(X), \quad
\wh\Pic(X)_{[\QQ]}=\wh\Pic(X)_{\QQ}.
$$
However, they are not true for $\RR$ or $\CC$. But we still have
$$\wh\Pic(X)_{[\CC]}=\wh\Pic(X)_{[\RR]}\otimes_\RR\CC.$$

Define $\wh\Pic(X)_{\intb, [\BFF]}$ to be the image of $\wh\Pic(X)_{\intb}\otimes_\ZZ\BFF$ in $\wh\Pic(X)_{[\BFF]}$.
The intersection theory extends to $\wh\Pic(X)_{\intb,[\BFF]}$ by linearity. 
The positivity notions are extended to $\wh\Pic(X)_{\intb,[\RR]}$.

The action $f^*: \wh\Pic(X)\to \wh\Pic(X)$ extends to 
$\wh\Pic(X)_{[\BFF]}$ naturally. 
The goal is to study the spectral theory of this action.

\begin{definition}
An element $\overline M$ of $\wh\Pic(X)_{[\CC]}$ is called \emph{$f$-admissible} if
we can write $\overline M=\sum_{i=1}^m \overline M_i$ such that each $\overline M_i$ is an eigenvector of $f^*$ in $\wh\Pic(X)_{[\CC]}$. 
\end{definition}

The main result here asserts the existence of an admissible section of the forgetful map $\wh\Pic(X)_{[\CC]}\to \Pic(X)_\CC$.

\begin{thm} \label{admisible metric}
For any $M\in \Pic(X)_{\CC}$, there exists a unique $f$-admissible lifting 
$\overline M_f$ of $M$ in $\wh\Pic(X)_{[\CC]}$. Moreover, for $\BFF=\ZZ, \QQ, \RR,\CC$, if $M\in \Pic(X)_{\BFF}$, then $\overline M_f\in \wh\Pic(X)_{[\BFF]}$.
\end{thm}

\begin{proof}
It suffices to assume that $M\in \Pic(X)_{\CC}$ is an eigenvector of $f^*$. Namely, $f^* M = \lambda M$ with $|\lambda|=q^{1/2}$ or $q$.
We claim that Tate's limiting argument of \cite{Zh3} still works here. 
Namely, let $\overline M$ be any lifting of $M$ in $\wh\Pic(X)_{[\CC]}$, then 
$(\lambda^{-1}f^*)^m \overline M$ converges \emph{uniformly} to a unique element in $\wh\Pic(X)_{[\CC]}$. This limit does not depend on the choice of $\overline M$, and we define $\overline M_f$ to be this limit.  
We only sketch an idea of the proof. 

Let $\overline M\in \wh\Pic(X)_{[\BFF]}$ be an element represented by 
$$
\overline M=\sum_{i=1}^r a_i\otimes \overline M_i, \quad a_i\in \BFF, \ \overline M_i\in \wh\Pic(X).
$$
Take \emph{an $\BFF$-section} $s$ of $\overline M$, i.e., a formal product 
$$
s= \otimes_{i=1}^r s_i^{\otimes a_i}, \quad s_i \in H^0(X,M)\setminus\{0\}.
$$ 
It has \emph{divisor}
$$
\div(s)=\sum_{i=1}^r a_i\ \div(s_i) \in \Div(X)_\BFF. 
$$
Then the metrics turn to a function 
$$
-\log \|s\| =- \sum_{i=1}^r a_i\log\|s_i\|
$$
in the space
$$S(X,\BFF)=\oplus_v S(X_{K_v}^\an,\BFF),$$
where $S(X_{K_v}^\an,\BFF)$ denotes the space of functions from $X_{K_v}^\an$ to $\BFF$ with logarithmic singularity along some $\BFF$-divisor. Then we convert the uniform convergence of $\overline M$ to the uniform convergence of $-\log\|s\|$. 

With these preparations, the original argument works here. We remark that we can further define the equivalence of two $\BFF$-sections $s_1$ and $s_2$ by $\div(s_1)=\div(s_2)$.
\end{proof}

\begin{example}
If $M\in \Pic ^0(X)$, then $\overline M_f$ is represented by an arithmetic class which is perpendicular to all vertical arithmetic classes. The arithmetic class was used in the proof of Theorem \ref{semisimplicity}.
\end{example}

\begin{example}
The admissible metrics for abelian varieties are well-known. 
Let $X$ be an abelian variety and $f=[m]$ be the multiplication by $m$ on $X$. Here $m$ is an integer with $|m|>1$. 
Any symmetric and ample line bundle $L$ on $X$ gives a polarization of $(X,f)$. We have $q=m^2$ in this case. 
Then $\ell_f(\NS(X)_\QQ)$ consists of exactly the rational multiples of symmetric line bundles, and $\Pic^0(X)$ is exactly the group of anti-symmetric line bundles. It is easy to see that $f^*$ acts as $m$ on $\Pic^0(X)$, and as $m^2$ on $\ell_f(\NS(X)_\QQ)$. 
The $f$-admissible classes can be obtained by the usual Tate's limiting argument (without get $\Pic(X)_\CC$ involved). 
\end{example}

Recall that we have a canonical section
$$
\ell_f: \NS(X)_\QQ \lra \Pic(X)_\QQ
$$
for the surjection $\Pic(X)_\QQ\to \NS(X)_\QQ$. 
By the $f$-admissible classes, we get a section 
$$
\wh\ell_f: \NS(X)_\QQ \lra \wh\Pic(X)_{[\QQ]}
$$
for the surjection $\wh\Pic(X)_{[\QQ]}\to \NS(X)_\QQ$. 

\begin{definition}
For  $\BFF=\QQ, \RR,\CC$, 
define 
$$
\wh\ell_f: \NS(X)_\BFF \lra \wh\Pic(X)_{[\BFF]}
$$
to be the map which sends $\xi\in \NS(X)_\BFF$
to the unique $f$-admissible class in $\wh\Pic(X)_{[\BFF]}$ extending $\ell_f(\xi)$. 
\end{definition}

\subsubsection*{Positivity}

The key result for our application is the following assertion. 

\begin{thm} \label{positivity}
Let $(X, f, L)$ be a {polarized dynamical system} over a number field $K$ with $X$ normal. 
If $M\in \Pic(X)_{\RR}$ is ample and $f$-pure of weight 2, then $\overline M_f$  is nef.
\end{thm}

\begin{proof}
By the action of the complex conjugation, we have $\overline M_f \in \wh\Pic(X)_{[\RR]}$. 

Let $\overline M$ be any nef extension of $M$ in $\wh\Pic(X)_{\nef}$. 
Still consider the sequence $\overline M_m=(q^{-1} f^*)^m \overline M$. 
Here every term $\overline M_m$ is nef. 
We will pick a subsequence ``convergent'' to $\overline M_f$. 
Decompose
$$
\overline M= \sum_{i=1}^r \overline N_i,
$$
where $f^* N_i=\lambda_i N_i$ with $|\lambda_i|=q$ for any $i$. 
Then
$$\overline M_m= \sum_{i=1}^r (q^{-1} f^*)^m \overline N_i.$$
To compare with 
$$
\overline M_{f}= \sum_{i=1}^r  \overline N_{i,f},
$$
write
$$\overline M_m= \sum_{i=1}^r  (q^{-1} \lambda_i)^m
\cdot (\lambda_i^{-1} f^*)^m \overline N_i.$$
Here $(\lambda_i^{-1} f^*)^m \overline N_i$ converges to $\overline N_{i,f}$ uniformly. 

Since $|q^{-1} \lambda_i|=1$, we can find an infinite subsequence $\{m_k\}_k$ such that $(q^{-1} \lambda_i)^{m_k} \to 1$ for every $i$. 
Then $\overline M_{m_k}$ ``converges'' to $\overline M$ in the sense of combining the uniform convergence of adelic metrics and the convergence of coefficients. 

Then the theorem follows from the lemma below. 
Note that $\overline N_{i,f}$ and $(q^{-1} f^*)^m \overline N_i$ may lie in $\wh\Pic(X)_{[\CC]}$ instead of in $\wh\Pic(X)_{[\RR]}$. But $\overline M, \overline M_f$ and $\overline M_m$ are in $\wh\Pic(X)_{[\RR]}$, and thus we can take the real parts of their decompositions above to get only elements of  $\wh\Pic(X)_{[\RR]}$ involved.
\end{proof}

\begin{lem}
Suppose we are given 
$$\overline M=\sum_{i=1}^r a_i \overline N_i, \quad
\overline M_m=\sum_{i=1}^r a_{i,m} \overline N_{i,m}, \quad
a_i\in \RR, \quad a_{i,m}\in \RR.$$ 
Here every $\overline N_i$ and every $\overline N_{i,m}$ are adelic line bundles on a projective variety $X$ over a number field. For any $i=1,\cdots, r$, assume the following convergence conditions:
\begin{itemize}
\item $a_{i,m}\to a_i$ as $m\to \infty$;
\item $N_{i,m}=N_{i}$ for any $m$;
\item $\overline N_{i,m}$ converges uniformly to $\overline N_{i}$ as $m\to \infty$.
\end{itemize}
If $M$ is ample and $\overline M_m$ is nef for any $m$, then $\overline M$ is nef. 
\end{lem}

\begin{proof}
Let $\overline M^0$ and $\overline N_i^0$ be any adelic line bundles extending $M$ and $N_i$, with some conditions we will impose later. Denote 
$$
\overline M_{m}^0= \sum_{i=1}^r  a_{i,m} \overline N_{i}^0.
$$
Let $\{\epsilon_m\}_m$ be a sequence in the interval $(0,1)$ convergent to 0. 
Consider
$$\overline M_m'= (1-\epsilon_m) \overline M_m -(1- \epsilon_m) \overline M_m^0+ 
\overline M^0.$$
Note that
$$\overline M_m'-\overline M= (1-\epsilon_m) (\overline M_m-\overline M_m^0) 
+(\overline M^0-\overline M).$$
We see that the underlying line bundle of $\overline M_m'$ is exactly $M$, 
and $\overline M_m'$ converges to $\overline M$ uniformly.
We are going to pick $\overline M^0, \overline N_i^0$ and $\epsilon_m$ so that  
$\overline M_m'$ is nef, which will prove the lemma. 
The conditions assume that $\overline M_m$ is nef. 
It suffices to make
$$
- (1-\epsilon_m) \overline M_m^0+ \overline M^0
=\epsilon_m \Big(\overline M^0+ (\epsilon_m^{-1}-1) (\overline M^0-\overline M_m^0)\Big)
$$
nef.

Let $\CX$ be an integral model of $X$. Let $\overline M^0$ and $\overline N_i^0$ be induced by Hermitian line bundles $\CMM$ and $\CNN_i$ on $\CX$. Denote
$$
\CMM_{m}= \sum_{i=1}^r  a_{i,m} \CNN_{i}.
$$
Assume 
$\CMM$ satisfies the following strong positivity conditions:
\begin{itemize}
\item $\CMM$ is ample in the arithmetic sense;
\item $\CM$ is ample in the geometric sense;
\item There is an embedding of $\CX(\CC)$ to a compact complex manifold 
$\Omega$, such that $\CM(\CC)$ can be extended to an ample line bundle on $\Omega$ and the metric of $\CMM$ can be extended to a positive (smooth) metric of the ample line bundle on $\Omega$. 
\end{itemize}
We further assume that the metric of $\CNN_i$ satisfies the regularity condition:
\begin{itemize}
\item For the same complex manifold $\Omega$ as above, the line bundle $\CNN_i(\CC)$ can be extended to a line bundle on $\Omega$ and the metric of $\CNN_i$ can be extended to a smooth metric of the line bundle on $\Omega$. 
\end{itemize}
These assumptions make sure that, 
in the (finite-dimensional) real vector subspace 
$V$ of $\wh\Pic(\CX)_{\RR}$ generated by $\CNN_1,\cdots, \CNN_r$ and $\CMM$, the subset of ample $\RR$-line bundles in $V$ form a neighborhood of $\CMM$. Here $V$ is endowed with the Euclidean topology. 
This is the open property of arithmetic ampleness and can be checked by the arithmetic Nakai--Moishezon criterion of Zhang \cite{Zh2}.

By definition, $\CMM-\CMM_m$ converges to $0$ in $V$. 
Hence, as long as $\epsilon_m$ converges to 0 much more slowly, the line bundle 
$$\CMM+ (\epsilon_m^{-1}-1) (\CMM-\CMM_m)$$
is ample on $\CX$. 
The proof is complete.

\end{proof}

\subsection{Preperiodic points}
The goal of this section is to prove Theorem
\ref{dynamicsmain}.

\subsubsection*{Canonical heights}

Let $(X,f,L)$ be a polarized dynamical system over a number field $K$ with $X$ normal.
Let $M$ be any line bundle on $X$. 
\emph{The canonical height} $h_{\overline M_f}: X(\overline K)\to \RR$ associated to $M$ is defined by 
$$
h_{\overline M_f}(x)= \frac{1}{\deg(x)} \wh\deg(\overline M_f|_{\tilde x}), \quad
x\in X(\overline K).
$$
Here $\tilde x$ denotes the closed point of $X$ representing $x$. 
The following are some basic properties:
\begin{itemize}
\item For any line bundle $M$, $h_{\overline M_f}(x)=0$ if $x$ is preperiodic.
\item 
If $M$ is ample, then $h_{\overline M_f}$ satisfies the Northcott property. 
\item 
If $M$ is ample and $f$-pure of weight 2, then $h_{\overline M_f}(x)\geq 0$ for any $x\in X(\overline K)$, since $\overline M_f$ is nef. 
\item For the polarization line bundle $L$, $h_{\overline L_f}(x)=0$ if and only if $x$ is preperiodic. It is an old result. 
\end{itemize}

More generally, we define \emph{the canonical height of a closed subvariety} 
$Y$ of $X$ associated to $M$ by 
$$
h_{\overline M_f}(Y):=\frac{1}{(\dim Y+1)\deg_M(Y)} \overline M^{\dim Y+1}.
$$
We still have the result that preperiodic subvarieties have canonical height 0. 
In fact, we have the following result. 
\begin{prop}
Let $Y$ be a preperiodic closed subvariety of $X$ of dimension $r$.
Let $M_1, M_2,\cdots, M_{r+1}$ be line bundles in $\Pic(X)_\CC$ which are $f$-pure of weight two. 
Then we have
$$
\overline M_{1,f}\cdot \overline M_{2,f}\cdots \overline M_{r+1,f}\cdot Y=0.
$$ 
\end{prop}
\begin{proof}
By the projection formula, we can assume that $Y$ is periodic. 
Replacing $f$ by a power if necessary, we can further assume that $f(Y)=Y$. 
By linearity, we can assume that $M_i$ is an eigenvector of $f^*$ for every $i$. 
Write $f^*M_i=\lambda_i M_i$. Then $f^*\overline M_{i,f}=\lambda_i \overline M_{i,f}$.
Then the projection formula gives
$$
f^*\overline M_{1,f}\cdot f^*\overline M_{2,f}\cdots f^* \overline M_{r+1,f}\cdot Y=
q^r\ \overline M_{1,f}\cdot \overline M_{2,f}\cdots \overline M_{r+1,f}\cdot Y.
$$ 
It follows that 
$$
(\lambda_1\cdots\lambda_{r+1}-q^r)\  \overline M_{1,f}\cdot \overline M_{2,f}\cdots \overline M_{r+1,f}\cdot Y=0.
$$ 
The result follows since $\lambda_1\cdots\lambda_{r+1}-q^r\neq 0$, as a consequence of the assumption $|\lambda_i|=q$.
\end{proof}

\begin{example}
The result is not true without the assumption that the line bundles are $f$-pure of weight two. Take $X$ to be an elliptic curve, and $M_1, M_2$ to be two line bundles of degree zero. Then they are $f$-pure of weight one if we take $f=[2]$. Then $\overline M_{1,f}\cdot \overline M_{2,f}=-2 \langle  M_1,M_2\rangle_{\rm NT}$ is often nonzero.  
\end{example}

\subsubsection*{Preperiodic points}

Now we are ready to prove the following theorem, which refines Theorem
\ref{dynamicsmain} in the case of number fields. The condition of $X$ being normal can be obtained by taking a normalization. 

\begin{thm} \label{dynamicsrefine}
Let $X$ be a normal projective variety over a number field $K$. For any $f,g\in \DS(X)$, the following are equivalent:
\begin{itemize}
\item[(1)] $\Prep(f)= \Prep(g)$;
\item[(2)] $g\Prep(f)\subset\Prep(f)$;
\item[(3)] $\Prep(f)\cap \Prep(g)$ is Zariski dense in $X$;
\item[(4)] $\wh\ell_f=\wh\ell_g$ as maps from $\NS(X)_\QQ$ to 
$\wh\Pic(X)_{\QQ}$.
\end{itemize}
\end{thm}

We will prove (1) $\Rightarrow$ (2) $\Rightarrow$ (3) $\Rightarrow$ 
(4)$\Rightarrow$ (1).
We start the proof with some easy directions.

First,  (1) $\Rightarrow$ (2) it trivial. 

Second,  (2) $\Rightarrow$ (3).
For any integer $d>0$, denote
$$\Prep (f,d):=\{x\in \Prep(f)\ |\ \deg(x)<d \}.$$
By Northcott's property, $\Prep (f,d)$ is a finite set since its points have
trivial canonical heights.
Assuming (2), then $g$ stabilizes the set $\Prep (f)$.
By definition,
$$\Prep (f)=\bigcup_{d>0} \Prep (f, d).$$
Since $g$ is also defined over $K$, it stabilizes the set $\Prep (f,d)$.
By the finiteness, we obtain that
$$\Prep (f,d)\subset \Prep (g), \ \forall d.$$
Hence,
$$\Prep (f)\subset \Prep (g).$$
Then (3) is true since $\Prep (f)$ is Zariski dense in $X$
by the result of Fakhruddin \cite{Fak}.

Third, we prove (4) $\Rightarrow$ (1).
Let $L$ be an ample line bundle on $X$ polarizing $f$. 
By (4), $\overline L_f=\overline L_g$. 
For any $x\in \Prep(g)$, we have $h_{\overline L_f}(x)=h_{\overline L_g}(x)=0$. 
It follows that $x\in \Prep(f)$. 
This proves $\Prep(g)\subset \Prep(f)$. 
By symmetry, we have the other direction and thus the equality.

\subsubsection*{From (3) to (4)}
Assume that $\Prep (f)\cap\Prep (g)$ is Zariski dense in $X$. 
Write $n=\dim X$. 
We need to prove $\wh\ell_f(\xi)=\wh\ell_g(\xi)$ for any $\xi\in\NS(X)_\QQ$. 
By linearity, it suffices to assume that $\xi$ is ample. 
Denote $L=\ell_f(\xi)$ and $M=\ell_g(\xi)$. They are ample $\QQ$-line bundles on $X$. Then $\overline L_f=\wh\ell_f(\xi)$ and $\overline M_g=\wh\ell_g(\xi)$ are nef by Theorem \ref{positivity}. 

Consider the sum $\overline N=\overline L_f+\overline M_g$, which is still nef.
By the successive minima of Zhang \cite{Zh3}, 
$$
\lambda_1(X,\overline N) \geq  h_{\overline N}(X) \geq 0.
$$
Here 
$$h_{\overline N}(X)=\frac{1}{(n+1)\deg_N(X)} \overline N^{n+1}
$$
and the essential minimum
$$
\lambda_1(X,\overline N)= \sup_{U\subset X} \inf_{x\in U(\overline K)}  h_{\overline N}(x),
$$
where the supremum is taken over all Zariski open subsets $U$ of $X$. 
Note that $h_{\overline N}$ is zero on
$\Prep (f)\cap\Prep (g)$, which is assumed to be Zariski dense in $X$.
Hence, $\lambda_1(X,\overline N)=0$. 
It forces $h_{\overline N}(X) = 0.$

Write in terms of intersections, we have 
$$
(\overline L_f+\overline M_g)^{n+1}=0.
$$
Expand by the binomial formula. 
Note that every term is non-negative. 
It follows that 
$$
\overline L_f^{i}\cdot\overline M_g^{n+1-i}=0, \quad\forall
i=0,1,\cdots, n+1.
$$
It particularly gives
$$
(\overline L_f-\overline M_g)^2 \cdot
(\overline L_f+\overline M_g)^{n-1}=0.
$$

Note that 
$$(L-M)\cdot (L+M)^{n-1}=0$$
since $L-M\in \Pic^0(X)_\QQ$ is numerically trivial. 
We are in the situation to apply Theorem \ref{hodge} to 
$$(\overline L_f-\overline M_g,\ \overline L_f+\overline M_g).
$$ 
It is immediate that $(\overline L_f-\overline M_g)$ is $(\overline L_f+\overline M_g)$-bounded.
The only condition that does not match the theorem is that $(\overline L_f+\overline M_g)$ is not ample. 
However, since $L-M$ is numerically trivial, 
as in the remark after the theorem,
we can take any $\overline C\in \wh\Pic(K)$ with $\deg(\overline C)>0$, and replace 
$$
(\overline L_f-\overline M_g,\ \overline L_f+\overline M_g)
$$
by 
$$
(\overline L_f-\overline M_g, \ \overline L_f+\overline M_g+\pi^*\overline C).
$$
Then all the conditions are satisfied. 
The theorem implies that 
$$
\overline L_f-\overline M_g \in \pi^*\wh\Pic(K).
$$

By evaluating at any point $x$ in $\Prep (f)\cap\Prep (g)$,
we see that 
$$
\overline L_f-\overline M_g =0
$$
in $\wh\Pic(X)_{\QQ}$. 
Here we have used the restriction $\overline L_f|_x=0$, which is more delicate than
$\wh\deg(\overline L_f|_x)=0$.
It finishes the proof.

\subsection{Variants and questions}

Now we consider some variants, consequences and questions related to Theorem \ref{dynamicsmain}.

\subsubsection*{Variants}

In a private communication, Barry Mazur points out that one direction of Theorem
\ref{dynamicsmain} can be generalized as follows:

\begin{thm}\label{dynamicssub}
Let $X$ be a projective variety over a number field $K$.
Let $f,g\in \DS(X)$, and denote by $Y$ the Zariski closure of $\Prep(f)\cap
\Prep(g)$ in $X$.
Then
$$\Prep(f)\cap Y(\overline K)= \Prep(g) \cap Y(\overline K).$$
\end{thm}

The proof of Theorem \ref{dynamicsmain} applies here. In fact, we can always restrict $f$-admissible (or $g$-admissible) adelic line bundles from $X$ to $Y$. Then we apply Theorem \ref{hodge} on $Y$.

One consequence of Theorem \ref{dynamicsrefine} is the following result. 

\begin{cor}\label{dynamicslocal}
Let $X$ be a projective variety over a number field $K$, and $f,g\in \DS(X)$ be two polarizable algebraic dynamical systems.
If $\Prep(f)\cap \Prep(g)$ is Zariski dense in $X$, then $d\mu_{f,v}=d\mu_{g,v}$ on $X_{K_v}^\an$ for any place $v$ of $K$.
\end{cor}

Here $d\mu_{f,v}$ denotes the equilibrium measure of $(X, f)$ on the analytic space $X_{K_v}^\an$. It can be obtained from any initial ``smooth'' measure on $X_{K_v}^\an$ by Tate's limiting argument. By a proper interpretation, it satisfies $f^*d\mu_{f,v}=q^{\dim X}d\mu_{f,v}$ and $f_*d\mu_{f,v}=d\mu_{f,v}$.

On can deduce the theorem by the equivalent condition $\wh\lambda_f=\wh\lambda_g$ in Theorem \ref{dynamicsrefine}. Alternatively, one can apply the equidistribution theorem of Yuan \cite{Yu} to any generic sequence in 
$\Prep(f)\cap \Prep(g)$ to obtain the result.

\subsubsection*{Semigroup}

For any subset $P$ of $X(\overline K)$, denote
$$
\DS(X,P):=\{g\in \DS(X)\ |\ \Prep(g)=P \}.
$$
We say that that $P$ is a \emph{special set of $X$} if $\DS(P)$ is non-empty.

\begin{question} 
Let $P$ be a special set of $X$.
Is the set $\DS(X,P)$ is a semigroup?
\end{question}

The question asks whether $g\circ h\in \DS(X,P)$ for any $g,h\in\DS(X,P)$.
By Theorem \ref{dynamicsmain}, we can write:
\begin{eqnarray*}
\DS(X,P)= \{g\in \DS(X)\ |\ gP\subset P \}.
\end{eqnarray*}
Then $P\subset \Prep(g\circ h)$ by
the simple argument proving (2) $\Rightarrow$ (3) of Theorem \ref{dynamicsmain}. 
Then we have $g\circ h\in\DS(X,P)$ if $g\circ h$ is polarizable. 

The polarizability is automatically true if $X$ is a projective space. 
Therefore, if $X=\BP^n$, then $\DS(X,P)$ is a semigroup.

The results and the questions also apply to general fields by the treatment of \cite{YZ}.

\subsubsection*{Dynamical Manin--Mumford}

The second author of this paper proposed in \cite{Zh3} a dynamical analogue of
the classical Manin--Mumford conjecture for abelian varieties. Namely, given a
dynamical triple $(X, L, f)$ over a field $K$ of characteristic zero,
a closed subvariety $Y$ of $X$ is preperiodic under $f$ if and only if
the set $Y(\overline K)\cap \Prep (f)$ is Zariski dense in $Y$.
Recently, Ghioca and Tucker found the following counter-example of this
conjecture:

\begin{prop}[Ghocia and Tucker, \cite{GTZ}]
 Let $E$ be an elliptic curve with complex multiplication by  an order $R$ in
imaginary quadratic field $K$.  Let $f$ be an endomorphism on  $E\times E$
defined by multiplications by two nonzero elements $\alpha$ and $\beta$  in $R$
with equal norm $\RN (\alpha)=\RN (\beta)$. Then $f$ is polarized by any
symmetric and ample line bundle
 $$\Prep (f)=E_{\tor}\times E_{\tor}.$$
 Moreover, the diagonal $\Delta _E$ in $E\times E$ is not preperiodic under $f$
if $\alpha/\beta$ is not a root of unity,
 \end{prop}

 Notice that the diagonal is preperiodic for multiplication by $(2,2)$.
 Thus the  proposition shows {\em an example that two endomorphisms of a
projective variety, with the same
set of preperiodic points, have different sets of preperiodic subvarieties.}  We
would like to propose the following revision of the dynamical Manin--Mumford
conjecture:

\begin{question}
Let $X$ be a projective variety over any field $K$, and $P$ be a special set of $X$. Let $Y$ be a proper
closed subvariety of $X$ such that $Y(\overline K)\cap P$ is Zariski dense in $Y$. Do there exist two endomorphisms $f,g\in\DS(X, P)$, and a proper $g$-periodic closed
subvariety $Z$ such that $f(Y)\subset Z$?
\end{question}

If the answers to both questions are positive, then in the situation of the second question, we can find a finite sequence of
subvarieties
$$X:=X_0\supset X_1\supset X_2\cdots \supset X_s$$
and endomorphisms $f_i, g_i\in \DS (X_i, P\cap X_i(\overline K))$ for
$i=0,1,\cdots, s$ such that
\begin{enumerate}
\item[(1)] $X_i$ is $g_i$-periodic, which implies that $P\cap X_i(\overline K)$
is a special set of $X_i$;

\item[(2)] $f_i\circ f_{i-1}\cdots f_0 (Y)\subset X_{i+1}$ for any
$i=0,1,\cdots, s-1$;

\item[(3)] $f_{s-1}\circ f_{s-2}\cdots f_0 (Y)= X_s$.
\end{enumerate}

\appendix

\section{Lefschetz theorems for Normal Varieties}

We list some classical Lefschetz-type results applicable to normal projective varieties over any characteristic. 
Their counterparts for complex projective manifolds are even more classical, and we refer  them to \cite[Chapter 3]{La}. 

Let $X$ be a projective variety of dimension $n\geq 2$ over an algebraically closed field $k$. Consider the exact sequence 
$$
0\lra \Pic^0(X) \lra \Pic(X)\lra  \NS(X)\lra 0.
$$
Here $\Pic^0(X)$ denotes the subgroup of algebraically trivial line bundles, and 
$\NS(X)$ denotes the quotient group. 

Recall that a line bundle $L$ on $X$ is numerically trivial if $L\cdot C=0$ for any closed curve $C$ in $X$.  
It is well-known that a line bundle $L$ is numerically trivial if and only if the multiple $mL$ is algebraically trivial for some nonzero integer $m$. See \cite[Theorem 6.3]{Kl}.

\begin{thm}\label{lefschetz2}
Let $L_1,\cdots, L_{n-1}$ be ample line bundles on $X$. 
For any $M\in \Pic(X)$ with $M\cdot L_1\cdots L_{n-1}=0$, one has
$$
M^2\cdot L_1\cdots L_{n-2}\leq 0.
$$
The equality holds if and only if $M$ is numerically trivial. 
\end{thm}

\begin{proof} 

If $X$ is a smooth projective surface, the result is the classical Hodge index theorem. See \cite[Theorem IV.1.9 ]{Ha} for example. If $X$ is a singular projective surface, the result is induced by a desingularization $X'\to X$. Note that the pull-back $L_1'$ of $L_1$ to $X'$ is not ample, but $L_1'^2=L_1^2>0$ is sufficient for the result. 

In general, by Bertini's theorem, we can assume that $L_1\cdot L_2\cdots L_{n-2}$ is represented by an integral closed surface $S$ in $X$. 
Then the inequality is proved by the Hodge index theorem on $S$. 

For the condition of the equality, we need to prove that $M\cdot C=0$ for any complete curve $C$ in $X$. 
By Bertini's theorem, we can assume that $L_1\cdot L_2\cdots L_{n-2}$ is represented by an effective $2$-cycle $\sum_{i=0}^r a_i S_i$ with $a_i>0$ such that $S_0$ contains $C$. 
Then 
$$\sum_{i=0}^r a_i S_i\cdot M^2=0$$ 
implies that $S_i\cdot M^2=0$
since each term is non-positive. 
It follows that $M|_{S_0}$ is numerically trivial on $S_0$. Then $M\cdot C=0$. 
\end{proof}

By a very ample linear system, we mean a subspace $V$ of $H^0(X,L)$, for a very ample line bundle $L$ on $X$, which gives an embedding 
$X\hookrightarrow \BP(V)$. Denote by $|V|=\{\div(s):s\in V\}$ the space of hyperplane sections. By a general hyperplane section of $V$ in $X$, we mean an element in a Zariski open subset of $|V|$. By the Bertini-type result of Seidenberg \cite{Sei}, if $X$ is normal and projective, then a general element $Y\in |V|$ is also normal and projective. 
The following is the Lefschetz hyperplane theorem in the current setting. 

\begin{thm} \label{lefschetz1}
Let $X$ be a normal projective variety of dimension $n$ over $k$. 
Let $Y$ be a general hyperplane section of a very ample linear system $V$ in $X$. 
\begin{enumerate}[(1)]
\item The natural map $\Pic^0(X)\to \Pic^0(Y)$ has a finite kernel if $n\geq 2$.
\item The natural map $\NS(X)\to \NS(Y)$ has a finite kernel if $n\geq3$. 
\item The natural map $\Pic(X)\to \Pic(Y)$ has a finite kernel if $n\geq3$. \end{enumerate}
\end{thm}
\begin{proof}
Part (3) is a consequence of (1) and (2). 
For (1), we refer to \cite[Remark 5.8]{Kl} for a historical account. 
Part (2) is a consequence of Theorem \ref{lefschetz2}. 
In fact, assume that $M$ lies in the kernel of $\NS(X)\to \NS(Y)$. 
In Theorem \ref{lefschetz2}, set $L_1=\cdots=L_{n-1}=\CO(Y)$. 
We see that $M$ is numerically trivial on $X$. 
Then some integer multiple of $M$ lies in $\Pic^0(X)$. 
Hence, the kernel of $\NS(X)\to \NS(Y)$ is a torsion subgroup. 
It must be finite since $\NS(X)$ is a finitely generated abelian group. 
\end{proof}

\begin{remark}
The theorems remain true if $X$ is projective and regular in codimension one, i.e., the singular locus $X_{\rm sing}$ has codimension at least $2$.
\end{remark}

\end{document}